\documentclass{amsart}
\usepackage{amsmath}
\usepackage{amsfonts}
\usepackage{amsthm, upref}
\usepackage{graphicx}
\usepackage{enumerate}
\usepackage[usenames, dvipsnames]{color}

\input xy

\usepackage{ifthen}
\usepackage{tikz}
\usepackage{appendix}
\usepackage{verbatim}
\usetikzlibrary{decorations.pathmorphing}
\usetikzlibrary{calc}
\usepackage{hyperref}
\usepackage{mathtools}

\renewcommand{\comment}[1]{}
\newcommand{\eq}{\begin{equation}}
\newcommand{\en}{\end{equation}}
\newcommand{\rr}{\mathbb{R}}
\newcommand{\NN}{\mathbb{N}}

\newcommand{\norm}[1]{\left\lVert #1 \right\rVert}

\newcommand{\iprod}[1]{\left\langle #1 \right\rangle }

\newcommand{\ol}[1]{\widetilde{#1}}

\newcommand{\Ar}[1]{\mathrm{Area}\left( #1  \right)}
\newcommand{\oD}{\overline{D}}
\newcommand{\uni}{\mathrm{Uni}}
\newcommand{\TV}{\mathrm{TV}}
\newcommand{\Prob}{\mathrm{P}}
\newcommand{\E}{\mathrm{E}}

\newcommand{\normU}[1]{\left\lVert #1 \right\rVert}
\newcommand{\skor}{\mathcal{D}^{(2m)}[0, \infty)}

\newcommand{\fil}{\mathcal{F}}

\newcommand{\rtoix}{\text{rank-to-index}}
\newcommand{\rtoshix}{\text{rank-to-shadow-index}}
\newcommand{\ixtoshix}{\text{index-to-shadow-index}}
\newcommand{\perm}{\mathcal{S}}
\newcommand{\fixed}{\mathbb{F}}

\newcommand{\sminx}{\mathfrak{i}}
\newcommand{\smjinx}{\mathfrak{j}}
\newcommand{\gands}{\text{gather-and-spread}}

\newcommand{\jdiffdist}{\mathcal{P}}
\newcommand{\diffdist}{\mathcal{R}}
\newcommand{\tmix}{t_{\mathrm{mix}}}

\begin{document}

\theoremstyle{plain}
\newtheorem{thm}{Theorem}
\newtheorem{lemma}[thm]{Lemma}
\newtheorem{prop}[thm]{Proposition}
\newtheorem{cor}[thm]{Corollary}

\theoremstyle{definition}
\newtheorem{defn}{Definition}
\newtheorem{asmp}{Assumption}
\newtheorem{notn}{Notation}
\newtheorem{prb}{Problem}

\theoremstyle{remark}
\newtheorem{rmk}{Remark}
\newtheorem{exm}{Example}
\newtheorem{clm}{Claim}

\title{Shuffling cards by spatial motion}

\author{Persi Diaconis}
\address{Departments of Mathematics and Statistics\\ Stanford University\\ Stanford, CA}
\email{diaconis@math.stanford.edu}

\author{Soumik Pal}
\address{Department of Mathematics\\ University of Washington\\ Seattle, WA}
\email{soumikpal@gmail.com}



\thanks{This research is partially supported by NSF grants DMS-1208775, DMS-1308340 and DMS-1612483}

\date{\today}

\begin{abstract}
	We propose a model of card shuffling where a pack of cards, spread as points on a square table, are repeatedly gathered locally at random spots and then spread towards a random direction. A shuffling of the cards is then obtained by arranging the cards by their increasing $x$-coordinate values. When there are $m$ cards on the table we show that this random ordering gets mixed in time $O\left(\log m\right)$. Explicit constants are evaluated in a diffusion limit when the position of $m$ cards evolves as an interesting $2m$-dimensional non-reversible reflected jump diffusion in time. Our main technique involves the use of multidimensional Skorokhod maps for double reflections in $[0,1]^2$ in taking the discrete to continuous limit. The limiting computations are then based on the planar Brownian motion and properties of Bessel processes. 
\end{abstract}

\maketitle

\section{Introduction}\label{sec:intro}

\subsection{The {\gands} model of spatial shuffling} Let $D=[0,1]^2$ represent a square table. Imagine $m$ labeled cards spread on this table. We will ignore the dimensions of the cards and represent them as particles with spatial positions in $[0,1]^2$.  Suppose at each discrete time step an agent selects a spot uniformly at random in $D$. Consider all cards whose current position lies in a disc of radius $\delta>0$ centered at that point. She gathers all such cards to the center of the disc in a single heap, randomly selects a direction, tosses an independent coin for each card in that heap, and, for those cards whose coins turn up heads, pushes the cards in that direction for a fixed, bounded distance while keeping them within the boundaries of the table. Other cards, including those whose coins turn up tails, are not moved. She does this independently at each time step. After $T$ steps the cards are projected on the $x$-axis and arranged in a line in the increasing order of the $x$-coordinate values. We are interested in the resulting random permutation of the set $[m]:=\{ 1,2,\ldots, m   \}$, especially in estimating $T$ that guarantees that this terminal random permutation is approximately uniform in total variation distance, irrespective of the initial positions of the cards. The model is designed to mimic a popular way to mix cards (called smooshing) by gathering cards locally using both palms and then spreading the cards by dragging them under the palm.

We now give a more formal definition. Let $U$ be the closed disc of radius $\delta$ centered at the origin. 
For any $z\in \rr^2$, the set $z+U$ will denote the disc centered at $z$. Colloquially we will refer to $U$ as the ``palm'' of the agent put at the point $z$.

\medskip

\begin{figure}[t]
	\includegraphics[width=7.5cm]{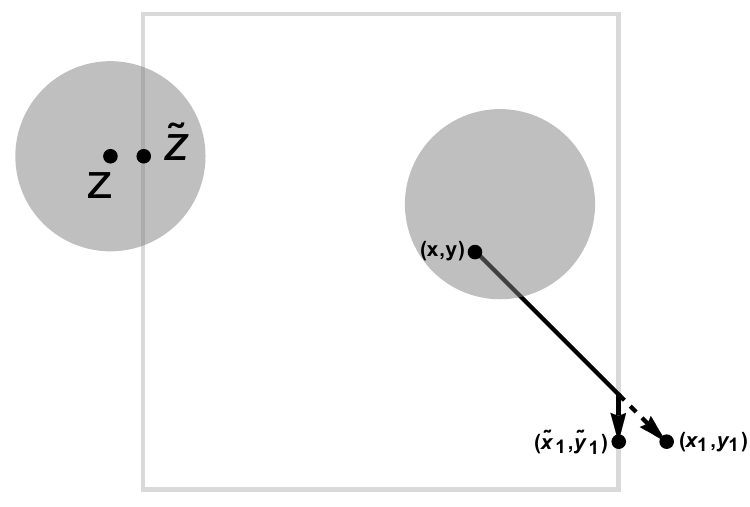}
	\caption{The effect of the boundary. On the left, the palm is centered at $z$ which is outside the table. All cards in the shaded area inside the square will then be gathered at $\widetilde{z}$. On the right we see spread (without gather). $(x,y)$ is a point under the palm which is not gathered at the center. The palm drags the card at $(x,y)$ along a straight line until the $x$-coordinate hits the boundary. The $x$-coordinate freezes while the $y$-coordinate continues to decrease.}
	\label{fig:gather_spread}
\end{figure}

We now describe the ``gather'' operation. Consider a point $z_0=(x_0,y_0)\in D$. Let $G^{z_0}:D \rightarrow D$ denote the map
\[
G^{z_0}(z):= z_0 1\left\{  z\in z_0 + U \right\} + z 1\left\{ z\notin z_0 + U \right\}, \qquad z=(x,y)\in D.
\] 
That is, points under the palm are gathered to the center. If $z$ is close to the boundary of $D$, there are fewer points to which it can be gathered. For tractability of our stochastic processes we will require some spatial homogeneity. This inspires the following extended definition. 

Let $\oD$ denote the Minkowski sum of the two sets $D$ and $U$. That is $\oD= \cup_{z\in D} \left\{ z + U    \right\}$. 
For $x\in \rr$, let 
\[
\ol{x}:=\max(\min(x,1),0)=\begin{dcases}
	0,& \text{if $x< 0$.}\\
	x,& \text{if $0\le x\le 1$.}\\
	1,& \text{if $x> 1$.}
\end{dcases} 
\]
For $z_0=(x_0, y_0) \in \oD \backslash D$, define $\widetilde{z}_0:= \left( \ol{x_0}, \ol{y_0}  \right)$ and
\[
G^{z_0}(z)=\widetilde{z}_0 1\left\{  z\in z_0 + U \right\} + z 1\left\{ z\notin z_0 + U \right\},\; z\in D.
\]
That is, points under the palm are gathered to a boundary point in case the center is outside $D$. See Figure \ref{fig:gather_spread} where the point $\mathbf{z}$ is outside the unit square and the corresponding $\mathbf{\tilde{z}}$ is on the boundary. Hence if the palm is placed such that the center is on $\mathbf{z}$, all cards under the palm will be gathered at $\mathbf{\tilde{z}}$.

We now define the ``spread'' operation. Fix $s_0 >0$, a $\theta \in [0, 2\pi]$, and a point $z_0=(x_0, y_0) \in \oD$.  
For $z=(x,y)\in D$ such that $z\in z_0+U$, let $x_1=x+s_0\cos(\theta)$ and $y_1=y+s_0 \sin(\theta)$. Define the map $f^{z_0, \theta}_{s_0}:D \rightarrow D$, by
\eq\label{eq:spread}
f^{z_0, \theta}_{s_0}(z) = \begin{cases} 
	\left( \ol{x_1}, \ol{y_1}   \right),& \text{if}\; z\in \{z_0+U\} \cap D.\\
	z, & \text{otherwise}.
\end{cases}
\en
Thus, for $z\in \{z_0+U\} \cap [0,1]^2$  (``cards under the palm'') we move $z$ linearly in the direction $\theta$ for distance $s_0$ until we hit the boundary of the table and stop moving the coordinate that is at the boundary. Nothing else is touched. 
Again, see Figure \ref{fig:gather_spread} where the point $(x,y)$, under the palm, is dragged until the $x$-coordinate hits one and does not increase any more. The $y$-coordinate, however, continues to decrease. $(x_1,y_1)$ represents the position of the particle had there been no boundary. The actual position is given by the coordinates $\left( \ol{x_1}, \ol{y_1} \right)$, where $\ol{x_1}=1$ since $x_1>1$.

If two or more cards cards occupy the same position (say due to gathering) we need additional randomization to break the ties. Fix $0 < p < 1$. Every time a card is about to be spread, it tosses an independent coin with probability $p$ of turning up heads. If it turns up heads, the card follows the palm in the chosen direction. Otherwise it does not move.  
We define this process formally below. 

Fix $\lambda> 0$. For mathematical convenience we consider continuous time and model the random selection of spots by the agent as a Poisson point process (PPP) on $(0, \infty)\times\oD$ of constant rate $\lambda$ with respect to the product Lebesgue measure. Since $\oD$ is bounded, it is possible to enumerate the atoms of this point process in a sequence $\left\{ \left(  t_i, w_i \right),\; i \in \NN    \right\}$ such that $t_1 < t_2 < t_3 < \ldots$, and each $w_i \in \oD$. One can obtain a discrete time model by discarding $t_i$s and considering the sequence $(w_i,\; i \in \NN)$ of i.i.d. uniformly chosen points in $\oD$ at discrete time points $i\in \NN$.

Let $\nu_0$ be any probability distribution on $[0,2\pi]$ that satisfies the following unbiasedness assumption. Here and throughout, $\nu(f(\cdot))$ for a probability measure $\nu$ and a function $f$, suitably measurable, will denote the expectation of $f$ under $\nu$.

\begin{asmp}\label{asmp:unbiased} Assume that $\nu_0\left( \cos(\cdot) \right)=\nu_0(\sin(\cdot))=\nu_0\left( \sin(\cdot)\cos(\cdot) \right)=0$ and that $\nu_0\left( \cos^2(\cdot) \right)=\nu_0\left( \sin^2(\cdot) \right)=\sigma^2$, for some $\sigma >0$.
\end{asmp}

Examples of $\nu_0$ include the uniform distribution over $[0,2\pi]$ and the discrete uniform distribution over the set $\{ 0, \pi/2, \pi, 3\pi/2 \}$ with $\sigma^2=1/2$ in both cases.

Generate an i.i.d. sequence $\left( \theta_i,\; i \in \NN \right)$ sampled from $\nu_0$ and consider the sequence of functions $\left( f_{s_0}^{w_i, \theta_i},\; i \in \NN  \right)$. Fix $m \in \NN$. Let $0< p<1$ be fixed as before.

\begin{figure}[t]
	\begin{tabular}{ll}
		\includegraphics[width=4.8cm]{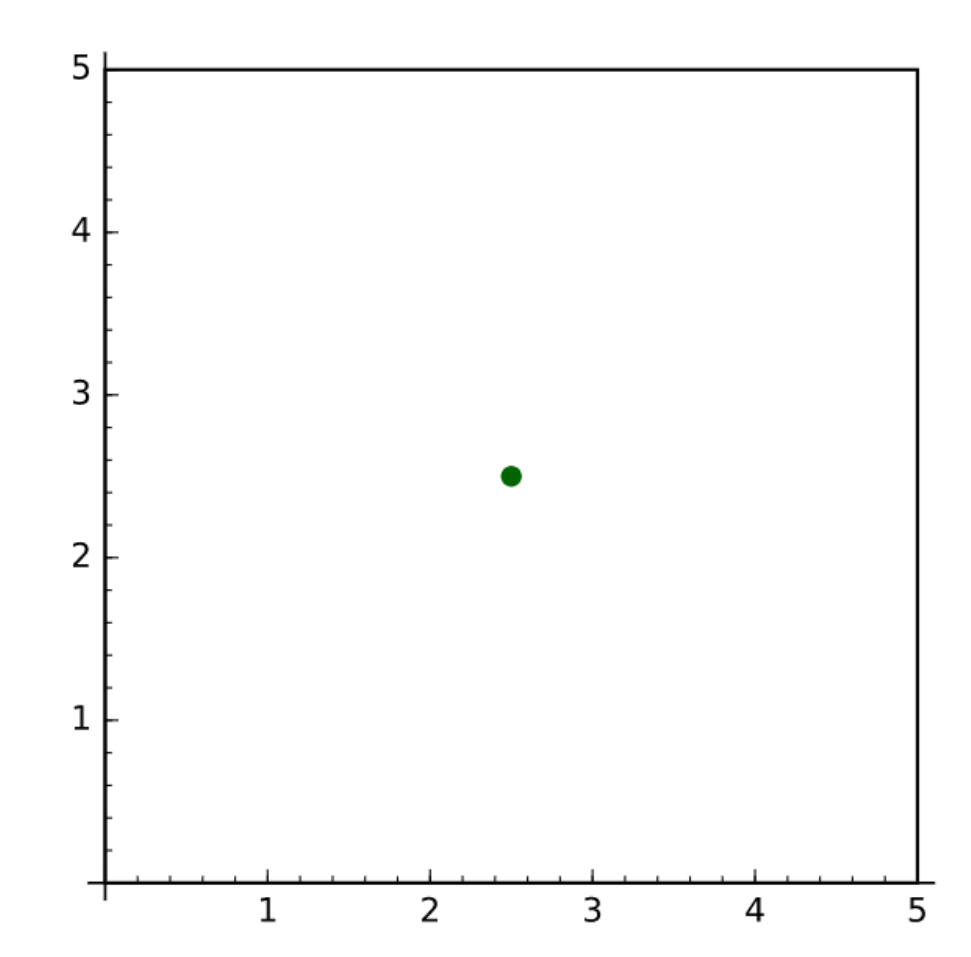}
		&
		\includegraphics[width=5.05cm]{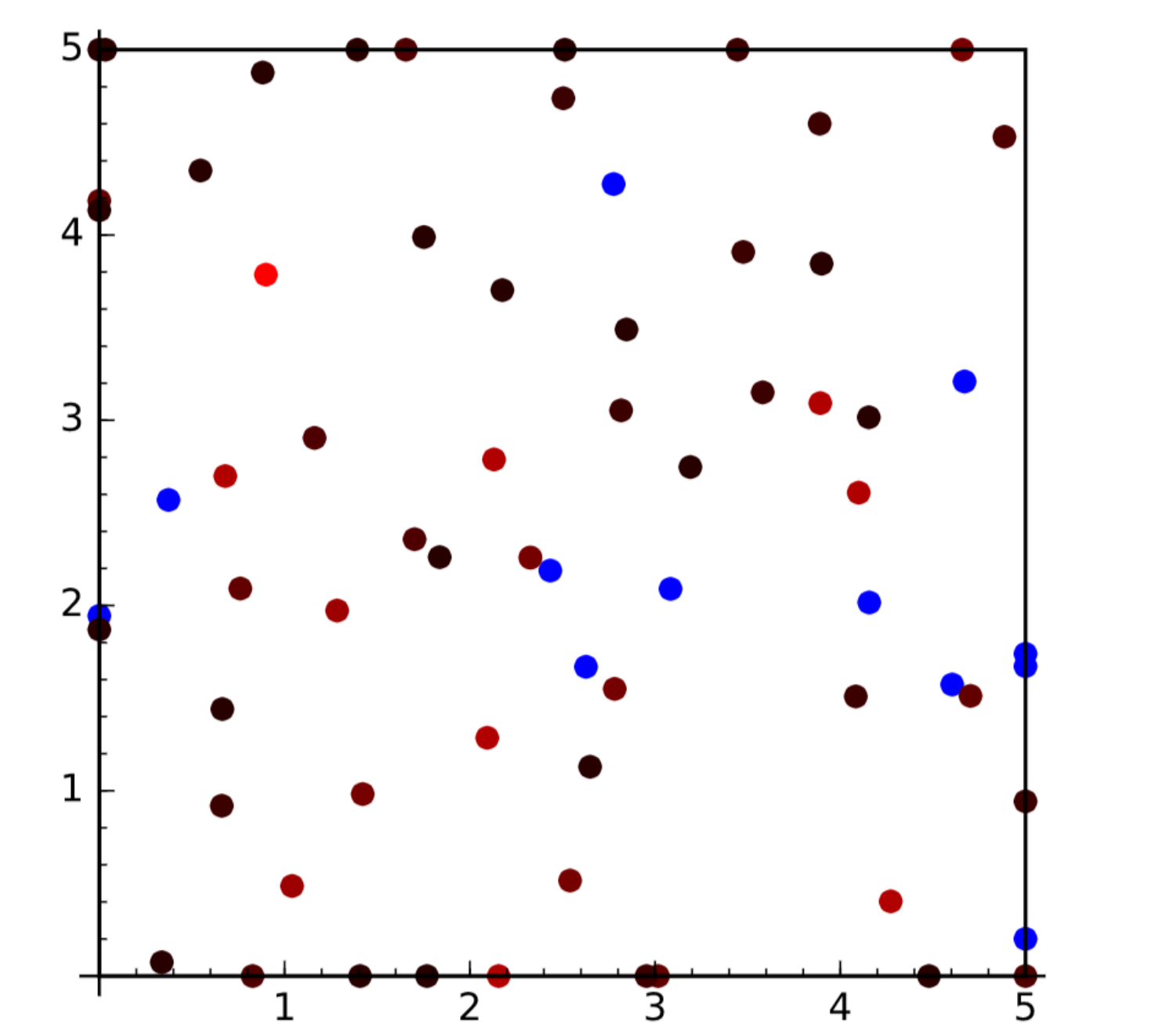}
	\end{tabular}
	\caption{\scriptsize{Configuration of $250$ cards after $3000$ steps of \gands. Table is $[0,5]^2$, $\delta=0.5$, $s_0=1$, $p=0.5$. Initially all cards are at the center shown on the left. At the terminal step, shown on the right, there are $68$ clusters of cards which are colored by the number of points in the cluster relative to the maximum. Singleton clusters are colored blue while the largest clusters is in bright red. There are $23$ clusters on the boundary. (Figure produced by Yuqi Huang)}.}
	\label{fig:simulation}
\end{figure}

\begin{defn}\label{defn:smooshingcards}(\textbf{$m$ point motion under {\gands}})  Let $Z_j(0)=z_j\in D$, $j\in [m]$, denote the initial positions of $m$ cards. Generate an {i.i.d.}\ array $\left( H_j(i), \; j\in [m], \; i\in \NN  \right)$ of Bernoulli($p$) random variables. Define sequentially, for $i\in \NN$, $j\in [m]$, starting with $t_0=0$,
	\eq\label{eq:gands}
	\begin{split}
		Z^0_j(t_i) &:= G^{w_i}\left( Z_j(t_{i-1})  \right), \quad \text{\textbf{(gather)} and} \\
		Z_j(t_{i}) &:= H_j(i) f_{s_0}^{w_i, \theta_i}\left( Z^0_j(t_{i})  \right) + \left(1 - H_j(i)\right)Z^{0}_j(t_{i}) , \quad \text{\textbf{(spread)}}.
	\end{split}
	\en
	Extend the sequence $\left( Z_j(t_i),\;j \in [m],\; i=0,1,2,\ldots  \right)$ to all $t \in [0, \infty)$ by defining 
	\eq\label{eq:rcllext}
	Z_j(t)= Z_j(t_{i-1}), \quad t_{i-1}\le t < t_i,\quad i\in \NN,\; j \in [m].
	\en
	The resulting continuous time Markov chain $\left( Z_j(t),\; j\in [m],  t\ge 0  \right)$ on the state space $D^{m}$ will be called the $m$ \textit{point motion under the {\gands} model}. 
\end{defn}

\begin{rmk}
	As before we choose to work with $\oD$ instead of $D$ for technical reasons. For a PPP on $(0, \infty) \times D$ of constant rate $\lambda$, the cards near the boundary move less frequently than those near the center of the table, affecting spatial homogeneity. Intuitively, the two models are not too different if $U$ is small compared to $D$.
\end{rmk}

Many natural variations of this model can be analyzed by the methods of this paper. For example, instead of gathering all the cards under the palm at the center one can choose a new independent, uniformly at random, position under the palm for each card. This is an example of \textit{local mixing}.  It is also possible to change the spread by selecting a probability distribution on $[0, \infty)$ and deciding the spread $s_0$ of each gathered card by sampling independently from it. For all such models the analysis in the paper remains similar. We will return to this point again.

We now describe what we mean by shuffling using this spatial motion. Let $\perm_m$ denote the group of permutations of $m$ labels. 
Consider an $m$-tuple of real numbers $(x_1, \ldots, x_m)$. Define the {\rtoix} permutation corresponding to this set in the following way. If every coordinate is distinct, then one can arrange the coordinates in increasing order $x_{(1)} < x_{(2)} < \ldots < x_{(m)}$, for a unique element $\gamma\in \perm_m$ such that $x_{\gamma_i}=x_{(i)}$. Say that the \textit{rank} of $x_{\gamma_i}$ is $i$ while the index of $x_{\gamma_i}$ is $\gamma_i$. When all coordinates are not distinct, the {\rtoix} will refer to a random variable taking values in $\perm_m$ which is obtained by ``resolving the ties at random''. To do this rigorously, generate {i.i.d.}\ uniform $[0,1]$ random variables $\{ U_1, \ldots, U_m \}$. Consider the set of pairs $\left\{ \left( x_1, U_1  \right), \ldots, \left(  x_m, U_m \right)   \right\}$. Rank the above sequence in the increasing dictionary order. That is $\left(  x_i, U_i \right) < \left( x_j, U_j\right)$ if, either $x_i < x_j$ or $\{ x_i=x_j \} \cap \{ U_i < U_j \}$. 
It is clear that, almost surely, this gives us a totally ordered sequence with no equalities. As before, let $\gamma \in \perm_m$ be the unique element such that $\left( x_{\gamma_i}, U_{\gamma_i}  \right)$ is the $i$th smallest element in the above ordering. Then $\gamma$ will be called the {\rtoix} permutation corresponding to $( x_1, \ldots, x_m )$. When $( x_1, \ldots, x_m )$ is random, the {\rtoix} permutation is the random permutation obtained by integrating with respect to its law. For example, if $m=4$ and $x_1=1, x_2=2, x_3=1, x_4=3$ then resolve the tie between $x_1$ and $x_3$ by generating i.i.d. Uni$(0,1)$ random variables $U_1, U_3$. Depending on whether $U_1 < U_3$ or $U_1 > U_3$, the {\rtoix} permutation would be either of the following (equally likely)
\[
\begin{pmatrix}
	1 & 2 & 3 & 4\\
	1 & 3 & 2 & 4
\end{pmatrix} \quad \text{or} \quad \begin{pmatrix}
	1 & 2 & 3 & 4\\
	3 & 1 & 2 & 4
\end{pmatrix}.
\]  

Consider the Markov chain $Z_j(\cdot)=\left( X_j(\cdot), Y_j(\cdot) \right)$, $j\in [m]$, from Definition \ref{defn:smooshingcards}. For any time $t \ge 0$, let $\gamma(t)$ denote the {\rtoix} permutation corresponding to the set of $x$-coordinates of the $m$ points, $\{ X_1(t), \ldots, X_m(t)  \}$.  Let $\norm{\gamma(t) - \uni}_{\TV}$ denote the total variation distance between the law of $\gamma(t)$ and the uniform distribution on $\perm_m$. By ``mixing time of shuffling'' we refer to the first time $\tmix(\epsilon)$ when this total variation distance is smaller than a given $\epsilon>0$, say $1/4$, irrespective of the vector of initial positions $\left(Z_j(0)=z_j,\; j \in [m]\right)$.

In Section \ref{sec:discrete} we provide an $O(\log m)$ mixing time bound for this Markov chain by developing a general scheme for all such problems. However, precise calculations of constants are not easy to derive for the discrete model. This difficulty is partly due to the existence of the boundary of $D$. Of course, projection of $Z_j$'s on the $x$-axis is arbitrarily chosen. By symmetry, the same bound holds for projection on the $y$-axis. 
Whether the problem of the boundary gets simplified by a more judicious choice of curve for projection is an interesting problem that is not answered here. 

In Section \ref{sec:diffusionlimit} the problem is simplified under a jump-diffusion limit as follows. First we will consider the parameter $\lambda \rightarrow \infty$ and $s_0=1/\sqrt{\lambda}\rightarrow 0$ while keeping $D$, $U$, and $\nu_0$ fixed. 
This means that we will make a lot of short spread moves. Furthermore, we will make gatherings rare by defining a \textit{lazy gathering model}. For each $t_i$ of the PPP in Definition \ref{defn:smooshingcards} we will toss an independent coin with a probability of heads given by $1/\lambda$. If the coin turns heads, we perform both gather and spread steps in \eqref{eq:gands}; otherwise we skip the gather step in \eqref{eq:gands} and only do the spread (with notations from \eqref{eq:gands}),
\eq\label{eq:gands2}
Z_j(t_i):= H_j(i) f_{s_0}^{w_i, \theta_i}\left( Z_j(t_{i-1})  \right) + \left(1 - H_j(i)\right)Z_j(t_{i-1}) , \quad i\in \NN, \; j\in [m].
\en    
Informally, per unit amount of time, we spread cards about $\lambda$ many times, each time by distance $O(1/\sqrt{\lambda})$, before gathering once. In fact, we gather at the jumps of a Poisson process of rate one. In between these jumps the Markov chain of $m$ point motion converges in law to a $2m$ dimensional diffusion with state space $D^{m}$ and reflected at the boundary. Thus the process evolves as a reflected diffusion that jumps according to a kernel at the points given by a Poisson process of rate one. This jump-diffusion is non-reversible and has reflections at the boundary of the non-smooth domain $D$. Hence information regarding its stationary distribution and rate of convergence cannot be inferred by standard methods. Nevertheless, we can bound the mixing time of shuffling thanks to our coupling scheme.   


\begin{thm}\label{thm:mixingtime}
	Under Assumption \ref{asmp:unbiased}, for all $m\ge 2$, the $m$-point motion under the lazy gathering model converges weakly in the Skorokhod space to a $2m$-dimensional reflected jump-diffusion model with state space $[0,1]^{2m}$ satisfying a stochastic differential equation described below in \eqref{eq:limitsde}.
	
	The mixing time of shuffling $m$ cards in this jump-diffusion model, $\tmix(\epsilon)$, is bounded above by   $\frac{1}{\mathfrak{p}}\log(m/\epsilon)$, where $0 < \mathfrak{p} < 1$ is given by
	\eq\label{eq:constantp}
	\mathfrak{p}:= \frac{\delta^2}{2p\pi\sigma^2 (1+2\delta)^2} K_0\left( \frac{\sqrt{2}+2\delta}{\sigma\delta \sqrt{\pi p (1-p)}}   \right).
	\en
	Here $K_0$ is the modified Bessel function of the second kind.
\end{thm}

Details about the modified Bessel function $K_0$ can be found in \cite[Section 9.6]{AbraSteg}. Notably, for a real argument $z$, the following asymptotics hold: $K_0(z)\sim -\log(z)$ as $z\rightarrow 0$, and $K_0(z)\sim \sqrt{\pi/(2z)} \exp(-z)$, as $z\rightarrow \infty$. 

The constant in \eqref{eq:constantp} is an increasing function of $\delta$, for a fixed choice of $p$ and $\sigma^2$. However, it is rather small and the bound is far from optimal. For example, suppose the palm is large enough to cover the entire table if placed at the center. That is, $2\delta > \sqrt{2}$. Then, after an exponential amount of time, the palm will gather all points under it which are then automatically uniformly shuffled by our tie-breaking rule. But, of course, this bound will not work for moderate to small $\delta$. On the other hand, for a choice of a moderate $\delta=0.3$ we get a minuscule $\mathfrak{p}\approx 1.88\times 10^{-7}$! Theorem \ref{thm:mixingtime} should be interpreted as simply an upper bound that is logarithmic in the number of cards with a large, but known, constant.

It is worthwhile to discuss some structural similarities between our card shuffling model and the motion of fluid under random stirring. The motion of a fluid is generally characterized by either an Eulerian or a Lagrangian description. The Eulerian description is provided by an explicit velocity field $v(t,x)$, which is the velocity that any fluid particle experiences at time $t$ if its position at that time is given by $x$ in an Euclidean space. The Lagrangian description $X(t)$ traces out the position of a single particle as a function of time. There is considerable literature on spatial mixing of fluids in non-stochastic settings. Here, a viscous liquid (e.g., molten glass) is considered and the behavior of a set of tagged particles (say, that of a dye) is studied. See Sturman et al \cite{sturman06} and Paul et al \cite{paulhandbook} for textbook accounts and Gouillart et al \cite{GouillartB, GouillartC, GouillartThesis} for recent advances. Our gather-and-spread operation is an Eulerian description in a stochastic setting with one important difference. Two fluid particles at the same position experience the same velocity field and will never separate. However, two cards at the same position can separate due to additional randomness. Nevertheless two such cards are \textit{exchangeable} in the sense that permuting their paths is a measure preserving operation. This is a critical feature of our model that is repeatedly invoked. The Markov chain of the positions of cards is then the corresponding Lagrangian motion. 

As mentioned before, gathering is a local mixing strategy. For example, imagine a viscous fluid (say, cake batter). We take a beater, randomly choose spots in the batter, and vigorously mix the location. If there are particles on the batter in that specific location, they will be so vigorously mixed as to be \textit{exchangeable} in their future evolution. Our results are valid for all such local mixing procedures. The spread, on the other hand is an ``advection-diffusion'' where we imagine a rod dipped in the fluid being dragged in a direction and creating a shear in its wake. Now, for best mixing practices in the non-stochastic setting it is intuitive to desire a chaotic system. This is usually achieved by repeating two perpendicular directions of shear with self-crossing trajectories (see Aref \cite{Aref}) such as a repeated figure eight movement through the fluid. This, along with diffusion in the fluid, cause mixing. In this sense, our {\gands} moves have been designed to study the effect of local mixing and a stochastic advection on particles in an underlying fluid.  

The analogy, however, breaks down in the meaning of the word ``mixing'', which is used in a different sense in fluid mixing. In that context, the points are unlabeled and we are interested in the difference between the empirical distribution of the points from the uniform distribution. This is not the case here. In fact, as shown in Figure \ref{fig:simulation}, there will always be clumps of points. In fact, it is not hard to see from our diffusion analysis that even for a single point, the uniform distribution is not the stationary distribution since the corners will have slightly more mass than the rest. Nevertheless we find this analogy motivating to further study both subjects.

\subsection{Review of literature}\label{sec:review} The mathematical study of shuffling cards has a long history going back to Poincar\'e \cite{Poincare}. Of course, this is a special case of the quantitative study of rates of convergence of Markov chains to their stationary distributions and random walks on groups. We recommend the book \cite{LPW} for an introduction and \cite{SC01, SC04} for a comprehensive overview. 

The present paper concerns \textit{spatial} mixing. Here, the literature is thinner and we offer a brief review. A crucial difference between the following literature and our model is that the stochastic process of permutations given by the $x$-coordinates of the cards in our model is not a Markov process by itself, but a function of an underlying Markov process given by the spatial positions of the cards. 

We start with the \textit{random adjacent transposition} chains. Picture $n$ labeled cards in a line, originally in order $1,2,\ldots,n$. At each step an adjacent pair of cards is chosen at random and the two cards are transposed. Results of Diaconis and Saloff-Coste \cite{DSC93} followed by Wilson \cite{Wilson04} show that order $n^3\log n$ steps are necessary and sufficient for convergence. Recently Lacoin \cite{Lacoin} sharpened this to show that there is a total variation cut-off at $n^3\log n /(2\pi^2)$.   

Random adjacent transpositions is a one-dimensional spatial model. One can extend the analysis to higher dimensions. For example, in two dimensions cards can be arranged on the vertices of a $\sqrt{n} \times \sqrt{n}$ grid. At each step an edge is chosen at random and the two cards at the vertices of this edge are transposed. This takes order $n^2\log n$ to mix \cite{DSC93}. These problems have become of recent interest as the `interchange process' because of their connections to suggestions of Dirac and Feynman in quantum mechanics. See \cite{ngthesis} for a tutorial and articles by Alon and Kozma \cite{AlonKozma} and Berestycki and Kozma \cite{BeresKozma} for interesting results.   

A related `mean-field' walk is the `random-to-random' walk. A randomly chosen card is removed and reinserted in a random position. In \cite{DSC93} order $n\log n$ steps are shown to be necessary and sufficient. In a tour-de-force \cite{DiekerSaliola} Dieker and Saliola determine all the eigenvalues and eigenvectors of this chain. A sharp cut-off has been recently established by Bernstein and Nestoridi in \cite{CutoffRandom}.    

A more overtly spatial walk is studied by Pemantle \cite{Pemantle94}. He considers $n^2$ cards at the vertices of an $n\times n$ array. At each step an element $x$ of the array is chosen uniformly at random. Then with probability $1/2$ the rectangle of cards above and to the left of $x$ is rotated $180^{\circ}$ degrees, and with probability $1/2$ the rectangle of cards below and to the right of $x$ is rotated $180^{\circ}$ degrees. While this is not a particularly natural model, it does have fascinating mixing properties. Pemantle shows that order $\Theta(n^2)$ steps are necessary and sufficient to mix all $n^2$ cards. However, for a fixed set of $k$ cards, $c(k)n$ steps suffice. Here, $c(k)$ is of order $k^3\left(\log k\right)^2$.

We conclude this review by reporting that we have also undertaken both simulations and a study of real world smooshing. Simulations on the gather-and-spread and related models were done by students at the University of Washington. The full report can be found at \cite{wxml}, from which Figure \ref{fig:simulation} is taken. At Stanford, a group of students smooshed for various times ($60$ seconds, $30$ seconds, $15$ seconds) with $52$ cards and $100$ repetitions for each time (so $100$ permutations for each of three times). A collection of ad-hoc test statistics were studied: position of original top (bottom) card, number of originally adjacent cards remaining adjacent, distance to the starting order in various metrics, length of the longest increasing subsequence, etc. The results suggest that randomness sets in after $30$ seconds or so while $15$ seconds was far from random. Since this kind of shuffling is used in both poker tournaments and Monte Carlo (see Diaconis et al \cite{DiaEvansGraham} for more on this), further study is of interest.

\section{Dimension consistency}\label{sec:warmup}

\subsection{A toy one-dimensional model}

Before we employ our coupling scheme in the general setting, it is instructive to use the same strategy in a simpler model. 
Fix arbitrary positive integers $m,N$. Consider $m$ labeled particles (representing cards) on a line. The position of each card can be one of the $N$ positive integers $[N]:=\{ 1,2,\ldots, N \}$.   

At time zero, the position of each particle is fixed, say, $X_i(0)$, $i \in [m]$. Time is discrete: $t=0,1,2,3,\ldots$. At each time $t$, pick a uniform random site, i.e., a random integer $k$ from the set $[N]$. Toss a fair coin to decide left or right. For each card that is currently at site $k$ (there may not be any), toss an independent coin with probability $p$ of turning up heads. 
\begin{itemize}
	\item If we decided left, move all cards whose coins turn heads to the left by one, if possible. That is, if $X_i(t)=k$ and the coin for $i$ turned heads, then $X_i(t+1)=(k-1)$, unless $X_i(t)=1$, in which case $X_i(t+1)=1$.  
	\item If we decided right, move all cards whose coins turn heads to the right by one, if possible. That is, if $X_i(t)=k$ and the coin for $i$ turned heads, then $X_i(t+1)=X_i(t)+1$, unless $X_i(t)=N$, in which case $X_i(t+1)=N$. 
	\item For all other cards $X_j(t+1)=X_j(t)$.  
\end{itemize} 
That is, imagine cards in boxes on a line, where $X_i(t)$ is the box of the $i$th card at time $t$, and cards pile on top of existing cards when they jump. But these details are not important mathematically. Repeat the above procedure by sampling a random site and the coin tosses at every time independent of the past. This gives us a stochastic process $X(t):=\left( X_1(t), X_2(t), \ldots, X_m(t)   \right)$, $t=0,1,2,\ldots$, which is the $m$ point motion on the line. 
At the end of time $T$, if we gather cards from left to right (breaking ties in any order), we get a shuffling of $m$ cards. 
We will estimate the mixing time of this shuffle using coupling.


Consider another pack of cards whose positions will be denoted by $X_i'(\cdot)$, $i=1,2,\ldots, m$. Critically, assume that the indices of the process $X'$ are assigned, uniformly at random, independent of their starting positions or that of the $X$s. This is true, for example, if $X'(0)$ is distributed according to the stationary distribution due to exchangeability of the coordinates. However, we do not need to assume that $X'(0)$ is distributed according to the stationary distribution for the following argument to hold. What is important is that, at any future point in time $t$, its rank-to-index permutation process $\gamma'(t)$ remains uniformly distributed over $\perm_m$, by the induced exchangeability irrespective of the initial positions of $X'$.    

Let ${Q}_m(x)$ denote the law of the process $X$ starting from $X(0)=x$ and let $\mathbb{Q}_m$ denote the law of the process $X'$ (under randomized indices). Now consider $2m$ cards whose positions will be denoted by $\left(X_1, \ldots, X_m, X_1', \ldots, X_m' \right)$. All $2m$ cards will now move according to the same choices of site selection and the direction to move. That is, both $X$ and $X'$ run using the same ``noise'', i.e., by the same realizations of uniform pick of random sites as well as the coin tosses to decide left or right.

Let $\tau_{i}$ denote the first time $t$ such that $X_i(t)=X_i'(t)$. Couple the two processes for all subsequent times by defining $X_i(t)=X'_i(t)$ for all $t \ge \tau_i$. Note that, for the individual coin tosses for the cards in order to decide if it should move or not, we use independent tosses until any $\tau_i$, after which $X_i$ and $X_i'$ are identified and will use the same coin toss. 

The crucial observation is that the the marginal distribution of the $m$-dimensional process $(X_1, \ldots, X_m)$ is still ${Q}_m(X(0))$, as if the other $m$ coordinates $(X_1', \ldots, X_m')$ do not exist. Similarly, the law of the other $m$-dimensional process $(X_1',\ldots, X_m' )$ is $\mathbb{Q}_m$. We will call this property \textit{dimension consistency} and will be shared by all the models in this paper. This property is a consequence of the fact that our models are point motions of stochastic flows of kernels that arose from the article by Harris \cite{Harris81}. See  \cite{BaxHarris86, LeJan85, meteor15, HowWarren09, LejanRaimond04} for subsequent developments in this theory. Hence, if we let $\tau^*=\max_i \tau_i$, then it follows that the total variation distance between $X(t)$ and $X'(t)$ is bounded above by $P(\tau^* > t)$. Thus, it suffices to estimate the tails of $\tau^*$. 

Consider again a system of $2m$ cards where two of them, say $X_1$ and $X_1'$, are initially at position $x$ and $y$, respectively. Let $\tau_{xy}$ denote the coupling time of these two cards. It follows from dimension consistency that the law of $\tau_{xy}$, under any initial positions of the $2m$ cards such that $X_1(0)=x, X'_1(0)=y$, is equal to its law under $Q_2(x,y)$.  

Suppose that there is a random variable $\zeta$ such that every $\tau_{xy}$ is stochastically dominated by $\zeta$, irrespective of $x$ and $y$. In the above example, suppose $1\le X_1(0)=x \le X_1'(0)=y$ without loss of generality. Both $X_1$ and $X_1'$ are lazy reflected random walks that jump at each turn with probability $p/N$ (probability $1/N$ for its site to be chosen and $p$ for its coin to turn up heads). Thus, irrespective of $x,y$, $\tau_{xy}$ is bounded by $\zeta^*$ which is the hitting time of $1$ for a lazy reflected random walk starting at $N$. The distribution of $\zeta^*$ can be found explicitly, but what is more important for our purpose is that there is another (possibly a different choice) random variable $\zeta$ such that every $\tau_{xy}$ is stochastically dominated by $\zeta$, and $\zeta$ has geometric tails. To see this note by the Markov property, 
\[
\begin{split}
	{Q}_2(x,y)&\left( \tau_{xy} > t + s \mid \tau_{xy} > t,\; X_1(t)=x', X_1'(t)=y' \right) \\
	&= {Q}_2(x', y')\left( \tau_{x'y'} > s\right) \le P(\zeta^* > s). 
\end{split}
\] 
Thus, ${Q}_2(x,y)\left( \tau_{xy} > t + s \mid \tau_{xy} > t \right) \le P(\zeta^* > s)$, and, hence
\[
{Q}_2(x,y)\left( \tau_{xy} > t + s \right) \le {Q}_2(x,y)\left( \tau_{xy} > t \right)  P(\zeta^* > s) \le P(\zeta^* > t)P(\zeta^* > s).
\]

Choose an $s$ such that $P(\zeta^* > s) \le 1/2$. Then, by iterating the above argument
\[
{Q}_2(x,y)\left( \tau_{xy} > t + k s \right) \le {Q}_2(x,y)\left( \tau_{xy} > t \right) \left( P(\zeta^* > s)\right)^k \le \frac{1}{2^k}. 
\]
Thus, there exists a constant $C$, depending on $s$ (say $C=2^{s}$), such that
\[
{Q}_2(x,y)(\tau_{xy} > t) \le C 2^{-t}, \quad \text{for all $t>0$},
\]
where we have chosen the parameter $t$ to be continuous for convenience. 


Now return to the scenario of $2m$ cards, $m$ of which have randomized indices. Recall that $\tau^*=\max_i \tau_i$. By a union bound, for any $t>1$,
\[
P\left( \tau^* > t \log_2 m \right) \le P\left( \cup_{i=1}^m \{\tau_i > t \log_2 m\} \right)  \le C m 2^{- t \log_2 m} \le C m^{-t+1}. 
\]
This gives an upper bound on the total variation distance between random vectors $X(t)$ and $X'(t)$. 

Consider from Section \ref{sec:intro} the rank-to-index permutation process $\gamma(t)$ corresponding to $X(t)$, and $\gamma'(t)$ corresponding to $X'(t)$. Recall that $\gamma'(t)$, for each $t$, is distributed uniformly over $\perm_m$. It follows (e.g. by choosing the same set of uniform random variables to break ranks) that the total variation distance between $\gamma(t)$ and $\gamma'(t)$ cannot be larger than that of $X(t)$ and $X'(t)$. Thus $\norm{\gamma(t) - \uni}_{\TV} \le C m^{-t+1}$. It is easy to see that the right side of the above inequality is $\epsilon$ if we choose $t=O(\log m)$. Thus the mixing time of shuffling $\tmix=O(\log m)$. 

We wish to repeat the fact that it is unimportant for this argument to assume that $X'$ is distributed according to the stationary distribution of the $m$-point motion. This is because we are only interested in the mixing time of shuffling which depends on the rank-to-index permutations and not on the spatial locations. However, this saves us the trouble of proving the existence of a stationary distribution for more complex models.

\subsection{A general bound for a triangular array of mutidimensional models}\label{sec:discrete} 
Although the {\gands} model is a two-dimensional model, all such fluid-dynamical models, irrespective of dimension, can be studied by a key consistency property that is described below.   

Throughout this paper, RCLL will refer to functions from $[0,\infty)\rightarrow \rr$ that are right continuous on $[0, \infty)$ and admit left limit at every point in $(0,\infty)$. For every $m\in \NN$, suppose that on a right-continuous and complete filtered probability space we have a strong Markov process $Z^{(m)}(t)=\left(Z^{(m)}_1(t), \ldots, Z^{(m)}_m(t)\right)$, $t\ge0$, where each component process has state space $\rr^d$, for some $d\ge 1$, and RCLL paths. In particular, we allow discrete time Markov processes by extending them to continuous time by piecewise constant interpolation. Let $Q_m(z_1, \ldots, z_m)$ denote the process law starting from the initial vector $(z_1, \ldots, z_m)$. At time $t$ extract the ``$x$-coordinates'' (which can be any of the $d$ coordinates) $X^{(m)}_1(t), \ldots, X^{(m)}_m(t)$ and consider the {\rtoix} permutation $\gamma^{(m)}(t)$. We are interested in the mixing time of shuffling for this permutation.

\begin{asmp}\label{asmp:exchangeable2}(\textbf{Dimension consistency.}) We say that the family of models $Q_m$, $m\in \NN$, satisfies the dimension consistency property if the following holds true. For any $m\in \NN$, let $Z^{(m)}$ be distributed according to $Q_m(z_1, \ldots, z_m)$. Then, the marginal distribution of the process $\left( Z^{(m)}_i, Z^{(m)}_j\right)$, for any $i\neq j$, is $Q_2(z_i, z_j)$. 
\end{asmp}

Note that the $m$ point motion under the {\gands} model in Definition \ref{defn:smooshingcards} satisfy this  property. In fact, recall the comment on the structural similarities between a fluid mechanics model and our card shuffling model as remarked in Section \ref{sec:review}. Assumption \ref{asmp:exchangeable2} will be valid for all such models.

For $x,y\in \rr^d$, let $\overline{\tau}^{(m)}_{x,y}$ denote the stopping time
\eq\label{eq:whatistaubar}
\overline{\tau}^{(m)}_{x,y}=\inf\left\{ t\ge 0:\; Z^{(m)}_i\left( t  \right)= Z^{(m)}_j(t) \right\}, \quad Z^{(m)}_i(0)=x,\; Z^{(m)}_j(0)=y.
\en
Note that $\overline{\tau}^{(m)}_{x,y}$ is indeed a stopping time due to our assumptions on the filtration. By Assumption \ref{asmp:exchangeable2}, its distribution under $Q_m(z_1, \ldots, z_m)$ does not depend on $m$ or $z_k$, $k\notin \{i,j\}$, and we will drop the $m$ from its notation and denote it by $\overline{\tau}_{x,y}$. 

\begin{asmp}\label{asmp:dominance}
	There is a positive random variable $\zeta^*$ such that $\overline{\tau}_{x,y}$ is stochastically dominated by $\zeta^*$ irrespective of $x,y$. 
\end{asmp}

\begin{thm}\label{thm:mixdiscrete}
	Suppose Assumptions \ref{asmp:exchangeable2} and \ref{asmp:dominance} hold. Then, there exists a positive constant $C$ such that for all $x,y \in \rr^d$ and all $m \in \NN$,   
	\eq\label{eq:mixdiscrete}
	P\left( \overline{\tau}_{x,y} > t \right) \le C 2^{-t}, \quad \text{for all $t\ge 0$}. 
	\en
	Moreover, if there is an $\alpha >0$ such that 
	\eq\label{eq:exptail}
	\liminf_{u\rightarrow 0+} \frac{1}{u} P\left( \zeta^* \le u \right) \ge \alpha,
	\en
	e.g. when $\zeta^*$ has a positive density $\alpha$ at zero, then every $\overline{\tau}_{x,y}$ is stochastically dominated by an exponential random variable with rate $\alpha$. 
	The mixing time for shuffling is $O(\log m)$ as $m\rightarrow\infty$, in particular, when \eqref{eq:exptail} holds, $\tmix(\epsilon)=\frac{1}{\alpha} \log(m/\epsilon )$.
\end{thm}

\begin{proof} Fix $m$. Let $G(t)=P\left( \zeta^* > t\right)=1 - P(\zeta^* \le t)$. 
	As in Section \ref{sec:warmup}, proceed by applying the strong Markov property: for any $t, s >0$, 
	\[
	P\left( \overline{\tau}^{}_{x,y} > t + s\right) \le P\left( \overline{\tau}^{}_{x,y} > t \right) G(s) \le G(t) G(s). 
	\] 
	Exactly as before \eqref{eq:mixdiscrete} follows. For the stronger conclusion \eqref{eq:exptail}, note that for a rational $t=k/n$, iterating the above gives us 
	\[
	P\left(\overline{\tau}^{}_{x,y} > k/n \right)\le \left(G(1/n) \right)^k=\left( 1 - P(\zeta^*\le 1/n) \right)^{nt}. 
	\]
	For any $0< \delta < \alpha$, if $n$ is taken large enough, by \eqref{eq:exptail}, $P(\zeta^*\le 1/n)\ge (\alpha-\delta)/n$. Thus
	\[
	P\left(\overline{\tau}^{}_{x,y} > t \right)\le \left( 1- \frac{\alpha-\delta}{n} \right)^{nt} \rightarrow e^{-(\alpha -\delta ) t},
	\]
	as $n\rightarrow \infty$. Now take $\delta \rightarrow 0+$ to get $P\left(\overline{\tau}^{}_{x,y} > t \right)\le \exp\left( -\alpha t \right)$ for all rational $t>0$. For irrational $t$, the inequality follows by right-continuity of the distribution function of $\overline{\tau}_{x,y}$. This proves that $\overline{\tau}^{}_{x,y}$ is dominated by an exponential$(\alpha)$ random variable. 
	
	The claim about mixing time of shuffling being $O(\log m)$ follows as before. Let us argue the special case when \eqref{eq:exptail} holds. In that, by stochastic domination by an exponential($\alpha$) variable, we have $Q_2(x,y)\left( \tau_{xy} > t \right) \le e^{-\alpha t}$. 
	
	Hence, as in the case of the toy one-dimensional model described in the last subsection, start now with $2m$ cards $(X, X')$ where the indices of $X'$ are assigned uniformly at random, independent of the rest. Let $\tau_i$ be the coupling time of cards $X_i$ and $X_i'$, and, let $\tau^*=\max_{i\in [m]} \tau_i$. One gets from the union bound 
	\[
	\norm{X(t) - X'(t)}_{\mathrm{TV}} \le P\left( \tau^* > t \right)\le me^{-\alpha t}.
	\]
	By equating the RHS to be $\epsilon$ and considering the pushforward to the rank-to-index permutations of both processes, as before, we get $\tmix(\epsilon)=\frac{1}{\alpha} \log(m/\epsilon )$.
\end{proof}

The mixing time of shuffling for the $m$ point motion under the {\gands} model would be $O(\log m)$ if a $\zeta^*$ exists. The difficulty in estimating the exact constant $\alpha$ is the boundary of the square table. The longer the cards stay at the boundary, the slower is the mixing. To do a finer analysis and working towards an invariance principle, we take a diffusion limit of the $m$-point motion. The jumps in the gathering are made rare and the spreads are made small and frequent. As we show in the next section, the sequence of processes converges to a jump-diffusion limit. This jump-diffusion spends a negligible amount of time at the boundary (a Lebesgue null set), and allows us to do a more precise estimate of the constants.

\section{The jump-diffusion limit of the {\gands} model}\label{sec:diffusionlimit} 

It is hard to get a precise estimate on the missing constant in the $O(\log m)$ mixing time bound above. In this section we treat a continuum jump-diffusion limit on the $m$ point motion of the lazy {\gands} model \eqref{eq:gands2} to estimate that constant by explicit computations facilitated by stochastic calculus.  

Consider a sequence of discrete {\gands} models with a corresponding sequence of parameters $\lambda^{(n)}=n$ and $s^{(n)}_0= 1/\sqrt{n}$, as $n\in \NN$ tends to infinity while keeping every other parameter ($U$, $\nu_0$, $0< p < 1$, initial values etc.) fixed. 

Recall the modification to Definition \ref{defn:smooshingcards} as given in \eqref{eq:gands2}. Let $(t_i, \theta_i,w_i,\; i\in \NN)$ and $\left(  H_j(i),\; j\in [m],\; i \in \NN \right)$ be as in Definition \ref{defn:smooshingcards}. Starting with $t_0=0$, at each $t_i$, toss a coin with probability of heads given by $1/\lambda=1/n$. If the coin turns heads, define $Z^{(n)}_j(t_i)$ for every $j\in [m]$ according to \eqref{eq:gands}. Otherwise, define
\eq\label{eq:updateti}
Z^{(n)}_j(t_i):= H_j(i) f_{1/\sqrt{n}}^{w_i, \theta_i}\left( Z^{(n)}_j(t_{i-1})  \right) + \left(1 - H_j(i)\right)Z^{(n)}_j(t_{i-1}) ,\; i\in \NN, \; j\in [m].
\en
Extending to all time periods as in \eqref{eq:rcllext} gives us a sequence of $m$ point motions
\[
Z^{(n)}:=\left(Z_j^{(n)}(t)=\left( X_j^{(n)}(t), Y^{(n)}_j(t)   \right),\; j \in [m],\; t\ge 0 \right),
\]
We are interested in the limit of this process as $n$ tends to infinity. 

Let $\varrho^{(n)}$ denote the first time we gather, i.e., the first time a coin with probability $1/n$ turns up heads. 
By the Poisson thinning property, this is distributed as a rate one exponential random variable (irrespective of $n$) and is independent of the process $Z^{(n)}(t)$, $0 \le t < \varrho^{(n)}$.

Consider the joint law of $\varrho^{(n)}$ and $Z^{(n)}(t)$, $0 \le t < \varrho^{(n)}$. As $n\rightarrow \infty$, we will show that the process $Z^{(n)}$ converges in law in the usual Skorokhod space to a continuous diffusion which is strong Markov, stopped at an independent exponential one time. This is enough for our purpose since by restarting the diffusion from a different initial condition (as dictated by the gather) at this random time gives us a limiting jump-diffusion. 

In order to describe the limiting (unstopped) diffusion we abuse our notations and assume that $Z^{(n)}$ is updated at every time $t_{i}$ by \eqref{eq:updateti} (without the gather) and take a diffusion limit. By an abuse of notation we continue to refer to this process (without the gather) in this section by $Z^{(n)}$ while keeping in mind that the process is observed only till an independent exponential time. 

Notice the following properties of the trajectories of cards under the spread moves (and no gather). Because of Assumption \ref{asmp:unbiased}, the increments of every card has mean zero, if it stays within the interior of the table after the spread. Also, if two cards are under the palm and they both decide to move with the palm, then their increments are positively correlated. However, if they are not under the palm, then the increment in one is independent of the other. This is captured by the fact that the limiting diffusion has zero drift (in the interior of the square) and a diffusion matrix given below. 

Define the $2\times 2$ positive-definite matrix
\[
\Sigma= \begin{bmatrix} \sigma^2 & 0 \\ 0 & \sigma^2 \end{bmatrix} = \begin{bmatrix}  \nu_0\left( \sin^2(\theta) \right) &  \nu_0\left( \sin(\theta)\cos(\theta)  \right)   \\   \nu_0\left( \sin(\theta)\cos(\theta)  \right) & \nu_0\left( \cos^2(\theta)  \right) \end{bmatrix}.   
\]
Note $\Ar{U}=\pi \delta^2$. Let $F$ be the $m \times m$ symmetric, positive definite matrix given by 
\eq\label{eq:whatisF}
F_{ij}( z_1, \ldots, z_m) = \begin{cases}
	p^2 \Ar{\{z_i + U \} \cap \{ z_j + U \}},& \quad  i,j\in [m], \; i\neq j.\\
	p \Ar{\{ z_i + U \}} = p \pi \delta^2, & \quad i=j \in [m].
\end{cases}
\en

We skip the proof of the following elementary fact. 

\begin{lemma}\label{lem:areaint} Let $z_1, z_2$ be two arbitrary points on the plane. Then 
	\[
	\Ar{\{z_1 + U \} \cap \{ z_2 + U \}}= \varphi\left( \norm{z_1-z_2}\right)
	\]
	where $\varphi:[0,\infty)\rightarrow [0, \infty)$ is given by
	\[
	\varphi(r):= \begin{dcases}
		2 \delta^2 \arccos\left( \frac{r}{2\delta}  \right) - \frac{r}{2}\sqrt{4\delta^2 - r^2},& \; \text{for $r\le 2\delta$},\\
		0,& \; \text{otherwise}. 
	\end{dcases}
	\]
	In particular, $\varphi$ is a decreasing convex function on $(0, \infty)$. 
\end{lemma}

Let $B(z_1, \ldots, z_m)$ be the $2m\times 2m$ matrix Kronecker product $F \otimes \Sigma$. See \cite[Definition 4.2.1]{HJ94} for the definition. Here and throughout we label the rows (and columns) of the diffusion matrix by assigning the $(2i-1)$th row to $x_i$ and $2i$th row to $y_i$, successively for $i=1,2,\ldots,m$. In particular, the block of $B$ corresponding to $z_i=(x_i,y_i)$ is a $2\times 2$ matrix given by $p\pi \delta^2 \Sigma$; the block corresponding to $(z_i=(x_i,y_i),z_j=(x_j,y_j))$th is a $4\times 4$ matrix given by 
\eq\label{eq:twobytwo}
\begin{bmatrix}
	p\pi \delta^2 \sigma^2 & 0 & F_{ij} \sigma^2 & 0\\
	0 & p\pi \delta^2 \sigma^2 & 0 & F_{ij} \sigma^2\\
	F_{ij} \sigma^2 & 0 & p\pi \delta^2 \sigma^2 & 0\\
	0 &  F_{ij} \sigma^2 & 0 & p\pi \delta^2  \sigma^2
\end{bmatrix}.
\en

By \cite[Theorem 4.2.12]{HJ94}, the eigenvalues of $B$ are pairwise products of those of $F$ and $\Sigma$. Hence, $B(z_1, \ldots, z_m)$, which is symmetric, is also nonnegative definite for any $(z_1, \ldots, z_m)$. We show later in Lemma \ref{lem:unifellip} that this matrix is uniformly positive definite. Let $A(z_1, \ldots, z_m)$ denote the unique positive definite square-root of $B(z_1, \ldots, z_m)$. Let $A_{X_j}(z_1, \ldots, z_m)$ and $A_{Y_j}(z_1, \ldots, z_m)$ denote the row of $A(z_1, \ldots, z_m)$ corresponding to coordinate $x_j$ and $y_j$, respectively, for $j \in [m]$. Thus, according to our convention, $A_{X_j}$ is the $(2j-1)$th row and $A_{Y_j}$ is the $(2j)$th row of $A$. 

\begin{defn}\label{defn:sdelimit}
	Fix points $ z_{j}=(x_j, y_j)\in [0,1]^2$, $j \in [m]$. Let $\diffdist_m(z_1, \ldots, z_m)$ denote the law of a time-homogeneous diffusion in $[0,1]^{2m}$ with zero drift, diffusion matrix $B$, normal reflection at the boundary, and initial conditions $Z_j(0)=z_j=(x_j, y_j)$, $j\in [m]$. The multidimensional vector-valued process $\left(Z_j(\cdot)=(X_j(\cdot),Y_j(\cdot))\right)$, $j\in [m]$, satisfies the stochastic differential equation (SDE):
	\eq\label{eq:limitsde}
	\begin{split}
		X_j(t)&= x_j + \int_0^t A_{X_j}(Z_1(s), \ldots, Z_m(s)) \cdot d\beta(s)  + L^{X,0}_j(t) - L^{X, 1}_j(t),\\
		Y_j(t)&= y_j + \int_0^t A_{Y_j}(Z_1(s), \ldots, Z_m(s)) \cdot d\beta(s)  + L^{Y,0}_j(t) - L^{Y,1}_j(t). 
	\end{split}
	\en
	Here $\beta=\left(  \beta_1, \ldots, \beta_{2m}  \right)$ is a $2m$ dimensional Brownian motion. The symbol $\int A_{\cdot} \cdot d\beta$ refers to the multidimensional It\^o stochastic integral with respect to the Brownian motion $\beta$. The processes $L^{X,0}_j(\cdot)$ and $L^{X,1}_j(\cdot)$ are the accumulated local times for the process $X_j$ at zero and one, respectively. The processes $L^{Y,0}_j(\cdot)$ and $L^{Y,1}_j(\cdot)$ are similarly defined. See \cite[Section 3.7]{KS91} for the normalization factor of local times.
\end{defn}

We show in Theorem \ref{lem:sdesoln} that, for $0 < p < 1$, the SDE \eqref{eq:limitsde} has a pathwise unique strong solution which is strong Markov. In particular, there is uniqueness in law and every solution is strong. The following is our main convergence result. Let $\skor$ be the usual Skorokhod space of RCLL paths from $[0,\infty)$ to $\rr^{2m}$. Unless otherwise mentioned, we work with the stronger locally uniform topology on this space. See \cite[Section 15 and 16]{Billingsley} for details on the Skorokhod space and the locally uniform and other topologies on it. This is for convenience. Since our limiting processes are continuous almost surely, the convergence with respect to the usual Skorokhod topology is equivalent to convergence in the locally uniform topology.

\begin{thm}\label{thm:jpoint} Fix $U$, $p\in (0, 1)$, and $\nu_0$. Fix an arbitrary set $\left\{  z_{j},\; j \in [m]  \right\}$ in $[0,1]^2$. Let $\left(Z^{(n)}_j(\cdot),\; j \in [m]\right)$ denote the $m$-point motion given in \eqref{eq:updateti} (without gather) when $\lambda^{(n)}=n$ and $s_0^{(n)}=1/\sqrt{n}$ starting with  $Z^{(n)}_j(0)=z_j$ for $j\in [m]$. Then, as $n$ tends to infinity, the sequence $\left( Z^{(n)}_j(t),\; t\ge 0, \; j \in [m]   \right)$, $n\in \NN$, converges in law in $\skor$ in the locally uniform topology to $\left( Z_j(t),\; t\ge 0, \; j \in [m]   \right)$ that is a solution of \eqref{eq:limitsde}. 
\end{thm}

That the limiting process should have zero drift and diffusion coefficients given by the matrix $B$ is easy to guess by computing the mean and the covariance of the increments of the discrete model. The appearance of local time is the consequence of the boundary behavior of our model and this is where it is critical that we use the $x\mapsto \ol{x}$ function in definition \eqref{eq:spread}.

Theorem \ref{thm:jpoint} is proved in several steps below. We start with $m=1$.

\begin{lemma}\label{lem:1pointlimit}
	Let $Z_1^{(n)}(t)=\left( X_1^{(n)}(t), Y_1^{(n)}(t) \right)$, $t\ge 0$, denote the one point motion given in \eqref{eq:updateti} (without gather) with $\lambda^{(n)}=n$ and $s_0^{(n)}=1/\sqrt{n}$ and given initial condition $Z_1^{(n)}(0)=(x,y)\in [0,1]^2$. Then, as $n\rightarrow \infty$, $\left(\left( X_1^{(n)}(t), Y_1^{(n)}(t) \right),t\ge 0 \right)$ converges in law to a pair $\left((X_1(t), Y_1(t)), \; t\ge 0\right)$ of independent doubly reflected Brownian motion (RBM) in the interval $[0,1]$ with zero drift and constant diffusion coefficient $p\pi \delta^2\sigma^2$, starting at $(x,y)$. In other words, $(X_1, Y_1)$ satisfies the SDE
	\eq\label{eq:sde2}
	\begin{split}
		X_1(t) &= x_1 + \sigma\delta \sqrt{p\pi} W_1(t) + L_1^{X,0}(t) - L_1^{X,1}(t), \\
		Y_1(t) &= y_1 + \sigma\delta\sqrt{p\pi} W_2(t) + L_1^{Y,0}(t) - L_1^{Y,1}(t),
	\end{split}
	\en
	where $(W_1, W_2)$ is a pair of independent standard one-dimensional Brownian motions, $L_1^{X,0}(t), L_1^{X,1}(t)$ are the accumulated local times at zero and one, respectively, till time $t$ for the semimartingale $X_1$, and $L_1^{Y,0}(t), L_1^{Y,1}(t)$ are similarly defined. 
\end{lemma}

The existence, uniqueness of reflected process and corresponding discrete to continuum convergence problems are best handled by the established tools of Skorokhod problems and Skorokhod maps. Skorokhod maps transform an ``un-reflected'' (or unconstrained) process to a ``reflected'' (or constrained) one inside a domain via a deterministic transform on the RCLL path-space. The existence of such a deterministic transform is called a Skorokhod problem. Once it is established that such a map exists, i.e. the Skorokhod problem has a solution, then existence, uniqueness, convergence problems can be handled ``pre-reflection'' by the usual martingale methods, and let the deterministic transform take care of the rest. See the exposition in \cite{Kruk07} where the reader can find more details.

To simplify the notation denote $\mathcal{D}^{(1)}[0,\infty)$, the Skorokhod space of RCLL functions from $[0, \infty)$ to $\rr$, by $\mathcal{D}[0,\infty)$. Let $\mathcal{BV}[0, \infty)$ and $\mathcal{I}[0, \infty)$ denote the subsets of $\mathcal{D}[0, \infty)$ comprised of functions of bounded variations and nondecreasing functions, respectively. 

\begin{defn}\label{defn:skorokhodmap} \textbf{Skorokhod map on {$[0,1]$}.} Given $\psi \in \mathcal{D}[0, \infty)$, there exists a unique pair of functions $\left( \bar{\phi}, \bar{\eta}  \right) \in \mathcal{D}[0, \infty) \times \mathcal{BV}[0, \infty)$ that satisfy the following two properties:
	\begin{enumerate}[(i)]
		\item For every $t\ge 0$, $\bar\phi(t)=\psi(t)+ \bar{\eta}(t)\in [0,1]$.
		\item $\bar{\eta}(0-)=0$ and $\bar{\eta}$ has the decomposition $\bar{\eta}=\bar{\eta}_l - \bar{\eta}_u$ as the difference of functions $\bar\eta_l, \bar\eta_u\in \mathcal{I}[0, \infty)$ satisfying the so-called complementarity conditions:
		\eq\label{eq:complementarity}
		\int_0^\infty 1\left\{ \bar\phi(s) > 0 \right\}d \bar\eta_l(s)=0 \quad \text{and}\quad \int_0^\infty 1\left\{ \bar\phi(s) < 1 \right \} d\bar\eta_u(s)=0.
		\en
	\end{enumerate}
	Here $\bar{\eta}(0-)=0$ means that, if $\bar\eta(0)>0$, then $d\bar\eta$ has an atom at zero. We refer to the map $\Gamma_{0,1}:\mathcal{D}[0, \infty) \rightarrow \mathcal{D}[0, \infty)$ that takes $\psi$ to $\bar\phi$ as the Skorokhod map on $[0,1]$. The pair $\left( \bar\phi, \bar\eta \right)$ is said to solve the Skorokhod problem on $[0,1]$ with input $\psi$.  
\end{defn}

The existence and uniqueness of Skorokhod map over general domains is a classical topic. See, for example, Tanaka \cite{Tanaka79}. Let $x^+=\max(x,0)$ for $x\in \rr$.
On $[0,1]$ the map has an explicit solution. In Theorem 1.4 of \cite{Kruk07} it is shown that $\Gamma_{0,1}=\Lambda_1 \circ \Gamma_0$, where 
\eq\label{eq:explicitskorokhod}
\begin{split}
	\Gamma_0(\psi)(t)&=\psi(t) + \sup_{0\le s \le t} \left[ -\psi(s)  \right]^+\quad \text{and}\\
	\Lambda_1(\phi)(t)&= \phi(t) - \sup_{0\le s \le t} \left[  \left( \phi(s)-1  \right)^+ \wedge \inf_{s\le u \le t} \phi(u)   \right].
\end{split}
\en
In particular, both $\Gamma_0$ and $\Gamma_{0,1}$ are Lipschitz with respect to the (locally) uniform and the Skorokhod $J_1$ metric on $\mathcal{D}[0, \infty)$.

\begin{proof}[Proof of Lemma \ref{lem:1pointlimit}] Fix $n\in \NN$. Let $\chi$ denote a PPP on $(0,\infty)\times \oD$ with rate $\lambda^{(n)}=n$. Evaluate the atoms of the PPP as a sequence $\left\{ (t_i,w_i),\; i\in \NN   \right\}$ where $t_i$ is increasing with $i$. Recall the i.i.d. sequence $\left( \theta_i,\; i \in \NN  \right)$ sampled from $\nu_0$ and an independent i.i.d. sequence of Bernoulli($p$) random variables $\left( H_i,\; i \in \NN  \right)$ from Definition \ref{defn:smooshingcards} (where we have substituted the notation $H_i$ for $H_1(i)$). Define $Z_1^{(n)}(\cdot)$ as in \eqref{eq:updateti}. 
	
	For $t\ge 0$, let 
	\[
	\begin{split}
		N(t)&:=\int_{[0,t] \times \oD}  1\left\{ Z_1^{(n)}(s-) \in  w+U \right\} d\chi(s,w) \\
		&= \sum_{j:\; t_j \le t} 1\left\{  \normU{Z_1^{(n)}(t_j-) - w_j} \le \delta \right\}.  
	\end{split}
	\]
	Then $N(t)$ counts the number of times the point $Z_1^{(n)}(\cdot-)$ is ``under the palm'' during time interval $[0,t]$. By the symmetry of the norm $\normU{\cdot}$ and the spatial homogeneity of the PPP, if the current position of the card is $z_1^{(n)}$, the first time when it is under the palm is an exponential random variable with rate $n\Ar{z_1^{(n)}+U}=n\pi \delta^2$, independent of the past. Thus, $\left(N(t),\; t \ge 0\right)$ is a Poisson process with rate $n\pi \delta^2$.
	
	Mark each jump time $t_i$ of $N$ with the corresponding $\theta_i$ and $H_i$. For $j \in \NN$, define
	\eq\label{eq:definexcounts}
	\begin{split}
		M^{X,n}(t_j)&= \frac{1}{\sqrt{n}} \sum_{i=1}^{j} H_i \cos(\theta_i).
	\end{split}
	\en
	Extend to other values of $t$ by defining that if $t_{j-1}\le t < t_j$ for some $j\in \NN$, then
	\eq\label{eq:definexcounts2}
	\begin{split}
		M^{X,n}(t)&= M^{X,n}(t_{j-1}).
	\end{split}
	\en
	Note that $M^{X,n}$ is a martingale since $\nu_0(\cos(\theta))=0$ and $H_i$ is independent of $\theta_i$.
	
	Now suppose that $\cos(\theta_j) \le 0$. Then the difference $X^{(n)}_1(t_j) - X^{(n)}_1(t_{j-1})$ is given by 
	\[
	\begin{dcases}
		\frac{1}{\sqrt{n}} H_j \cos(\theta_j),& \; \text{if}\; X^{(n)}_1(t_{j-1}) + \frac{1}{\sqrt{n}} H_j \cos(\theta_j) >0,\\
		- X^{(n)}_1(t_{j-1}), &  \text{if}\; X^{(n)}_1(t_{j-1}) + \frac{1}{\sqrt{n}} H_j \cos(\theta_j) \le 0.
	\end{dcases}
	\]
	We can express this differently as 
	\[
	X^{(n)}_1(t_j) - X^{(n)}_1(t_{j-1})= \frac{1}{\sqrt{n}} H_j \cos(\theta_j) + \left(X^{(n)}_1(t_{j-1}) + \frac{1}{\sqrt{n}} H_j \cos(\theta_j)\right)^-,
	\]
	where $x^-:=\max(-x,0) \ge 0$. 
	
	Similarly, when $\cos(\theta_j) > 0$, $X^{(n)}_1(t_j) - X^{(n)}_1(t_{j-1})$ is given by 
	\[
	\begin{dcases}
		\frac{1}{\sqrt{n}} H_j \cos(\theta_j),& \; \text{if}\; X^{(n)}_1(t_{j-1}) + \frac{1}{\sqrt{n}} H_j \cos(\theta_j) <1,\\
		1- X^{(n)}_1(t_{j-1}), &  \text{if}\; X^{(n)}_1(t_{j-1}) + \frac{1}{\sqrt{n}} H_j \cos(\theta_j) \ge 1.
	\end{dcases}
	\]
	Hence,
	\[
	X^{(n)}_1(t_j) - X^{(n)}_1(t_{j-1})= \frac{1}{\sqrt{n}} H_j \cos(\theta_j) - \left(1-X^{(n)}_1(t_{j-1}) - \frac{1}{\sqrt{n}} H_j \cos(\theta_j)\right)^-.
	\]
	
	Combining the two cases note that we can always write
	\[
	\begin{split}
		&X^{(n)}_1(t_j)  - X^{(n)}_1(t_{j-1}) = \frac{1}{\sqrt{n}} H_j \cos(\theta_j)\\
		& + \left(X^{(n)}_1(t_{j-1}) +\frac{1}{\sqrt{n}} H_j \cos(\theta_j)\right)^- - \left(1-X^{(n)}_1(t_{j-1}) - \frac{1}{\sqrt{n}} H_j \cos(\theta_j)\right)^-.
	\end{split}
	\]
	
	Define the following pair of increasing functions, both starting at zero:
	\eq\label{eq:disloctime}
	\begin{split}
		I^{X,0,n}(t_j) &= I^{X,0,n}(t_{j-1}) + \left(X^{(n)}_1(t_{j-1}) +\frac{1}{\sqrt{n}} H_j \cos(\theta_j)\right)^-\\
		I^{X,1,n}(t_j) &= I^{X,1,n}(t_{j-1}) + \left(1-X^{(n)}_1(t_{j-1}) - \frac{1}{\sqrt{n}} H_j \cos(\theta_j)\right)^-.
	\end{split}
	\en
	Extend them to all values of $t\in [0, \infty)$ by defining $I^{X,0,n}(t)=I^{X,0,n}(t_{j-1})$ and $I^{X,1,n}(t)=I^{X,1,n}(t_{j-1})$, for all $t\in [t_{j-1}, t_j)$. Note that the jumps of the process $I^{X,0,n}$ occur at those $t_j$ such that $X_1^{(n)}(t_j)=0$ and $\cos(\theta_j) < 0, H_j=1$ while the jumps of $I^{X,1,n}$ occur at those $t_j$ such that $X_1^{(n)}(t_j)=1$ and $\cos(\theta_j)>0, H_j=1$.

	Recall the martingale $M^{X,n}$ from \eqref{eq:definexcounts}. 
	From here it is not hard to see that $X_1^{(n)}(\cdot)$ is the solution of the following system of pathwise equations \eqref{eq:disloctime} and
	\[
	\begin{split}
		X^{(n)}_1(t) &= x_1 +  M^{X,n}\left( t\right) \\
		&+ \int_0^t 1\left\{ X^{(n)}_1(s) = 0 \right\} d I^{X,0,n}\left( s \right)- \int_0^t 1\left\{ X^{(n)}_1(s)=1  \right\} dI^{X,1,n}\left(s\right). 
	\end{split}
	\]
	This is an expression that satisfies the Skorokhod problem decomposition given in Definition \ref{defn:skorokhodmap}.
	The process $X^{(n)}_1$ is constrained to stay in $[0,1]$, $M^{X,n}$ is RCLL, while 
	\[
	\bar{\eta}_l:=\int_0^t 1\left\{ X^{(n)}_1(s) = 0 \right\} d I^{X,0,n}\left( s \right),\; \bar{\eta}_u:=\int_0^t 1\left\{ X^{(n)}_1(s)=1  \right\} dI^{X,1,n}\left(s\right)
	\]
	are increasing and obviously satisfy the complementarity conditions \eqref{eq:complementarity}. Hence $$\left( {X}_1^{(n)},  I^{X,0,n} - I^{X,1,n} \right)$$ is the unique solution of the Skorokhod problem on $[0,1]$ with input $x+{M}^{X,n}(\cdot)$.
	
	Now take limits as $n$ tends to infinity. It follows from Donsker's invariance principle that the continuous time martingale ${M}^{X,n}$ converges to $\sigma\delta\sqrt{p\pi} W_1$, where $W_1$ is a standard Brownian motion. This is because $M^{X,n}$ is a continuous time centered random walk that jumps at rate $n\pi\delta^2$ and the variance of its increments is $p\sigma^2/n$. By the Lipschitz continuity of the deterministic Skorokhod map, it immediately follows that the vector of processes
	\[
	\left( {M}^{X,n}, {X}_1^{(n)},  I^{X,0,n}, I^{X,1,n} \right)
	\]
	jointly converges in law to the vector of $\sigma\delta\sqrt{p\pi} W_1$ and the corresponding terms in the solution of the Skorokhod problem in $[0,1]$ with input $x_1+ \sigma\delta\sqrt{p\pi} W_1$.
	
	Let us now identify the limit as reflecting Brownian motion in the interval $[0,1]$ with constant diffusion coefficient $\sigma^2 p \pi \delta^2$. The limit, say $X_1$, satisfies the SDE given by the Skorokhod equation:
	\[
	X_1(t) = x_1 + \sigma \delta\sqrt{p\pi} W_1(t) + L^{X,0}(t) - L^{X,1}(t),
	\]
	where $L^{X,1}$ and $L^{X,0}$ are outputs from the Skorkhod problem with input $x + \sigma\delta \sqrt{p\pi} W_1(t)$. To identify $L^{X,1}$ and $L^{X,0}$ with the local time of the process $X_1$ at the boundary zero and one, respectively, we apply the Tanaka (\cite[page 220]{KS91}) formula to the semimartingale $X_1$ for the functions $x\mapsto x^+$ and $x\mapsto (1-x)^+$.
	\bigskip
	
	For the $y$-coordinate process repeat the above argument except that $\cos(\theta_j)$ will be replaced by $\sin(\theta_j)$. 
	That is, define
	\eq\label{eq:definexcounts3}
	\begin{split}
		M^{Y,n}(t_j)&= \frac{1}{\sqrt{n}} \sum_{i=1}^{j} H_i \sin(\theta_i).
	\end{split}
	\en
	Extend to other values of $t$ by keeping the process constant in each interval $[t_{j-1}, t_j)$, $j\in \NN$.
	$M^{Y,n}$ is also a martingale since $\nu_0(\sin(\theta))=0$.
	
	The naturally defined corresponding processes for the $y$-coordinate
	\[
	\left( {M}^{Y,n}, Y_1^{(n)}, I^{Y,0,n},  I^{Y,1,n} \right)
	\]
	jointly converges in law to the vector of a Brownian motion $\sigma \delta\sqrt{p\pi} W_2$ and the corresponding terms in the solution of the Skorokhod problem in $[0,1]$ with input $y+ \sigma \delta\sqrt{p\pi} W_2$. We need to argue joint convergence of the vector
	\[
	\left({M}^{X,n}, X_1^{(n)}, I^{X,0,n}, I^{X,1,n}, {M}^{Y,n}, Y_1^{(n)}, I^{Y,0,n}, I^{Y,1,n} \right).
	\]
	However, this will follow from the joint convergence of the pair $\left( {M}^{X,n}, {M}^{Y,n} \right)$ since everything else is a deterministic Lipschitz function applied to this pair of processes.

	We first claim that $\left( {M}^{X,n}(t){M}^{Y,n}(t),\; t\ge 0\right)$ is also a martingale. Since each process individually is a process of identically distributed independent increments, it suffices to check that the increments are uncorrelated. However, that is guaranteed by the Assumption \ref{asmp:unbiased} that $\nu_0\left( \cos(\theta)\sin(\theta) \right)=0$.
	
	Now, by marginal convergence, it follows that the sequence of laws of the pair of processes $\left((p\pi \sigma^2\delta^2)^{-1/2}{M}^{X,n}, (p\pi \sigma^2\delta^2)^{-1/2}{M}^{Y,n} \right)$ in $\mathcal{D}^{(2)}[0, \infty)$ is tight in the locally uniform metric and that any limiting processes $W_1, W_2$ are marginally Brownian motions that additionally satisfy $W_1W_2$ is a local martingale. It follows by Knight's theorem (see \cite[page 179]{KS91}) that $W_1, W_2$ must be a pair of independent Brownian motions. Since $X_1$ and $Y_1$ are outputs of the deterministic Skorokhod map applied to $W_1$ and $W_2$, they too are independent. This completes the proof.
\end{proof}

\begin{proof}[Proof of Theorem \ref{thm:jpoint}] This proof is a generalization of the proof of Lemma \ref{lem:1pointlimit}. As in that proof, for every $n\in \NN$, $j \in [m]$, and $i\in\{0,1\}$, define the quantities
	\[
	{X}_j^{(n)}, {Y}_j^{(n)}, M_j^{X,n}, M_j^{Y,n}, I_j^{X,i,n}, I_j^{Y,i,n}.
	\]
	Then, for each $j\in [m]$, the vector $\left( {X}_j^{(n)}, {Y}^{(n)}_j\right)$ can be expressed as the solution of a system of Skorokhod equations in $[0,1]$ with given inputs
	\eq\label{eq:multiinput}
	x_j + {M}_j^{X,n} \quad \text{and}\quad y_j + {M}_j^{Y,n}\left( \cdot \right), \quad \text{respectively}.
	\en
	
	The strategy is now the following. Consider the vector of $4m$ many processes obtained by concatenating $\left( X_j^{(n)}, Y^{(n)}_j,\; j \in [m] \right)$ with the $2m$ many inputs in \eqref{eq:multiinput}. Each coordinate process is tight by Lemma \ref{lem:1pointlimit} and has an almost sure continuous limit. Hence the joint law of these $4m$ processes is tight in $\mathcal{D}^{(4m)}[0,\infty)$, with the locally uniform metric, and any weak limit is a probability measure on $\mathcal{C}^{(4m)}[0, \infty)$. The latter is the space of all continuous functions from $[0, \infty)$ to $\rr^{4m}$ equipped with the locally uniform metric. 
	
	Let the $4m$ dimensional vector
	\eq\label{eq:limitproc}
	\left( X_j, Y_j, x_j + {M}_j^X, y_j + {M}_j^Y,\;  j\in [m]  \right)
	\en
	denote a process whose law is any weak limit of the sequence of processes 
	\eq\label{eq:limittilde}
	\left( X_j^{(n)}, Y^{(n)}_j, x_j + {M}^{X,n}_j, y_j + {M}^{Y,n}_j,\; j \in [m]   \right),\quad n \in \NN.
	\en

	
	To prove the existence of the limiting SDE representation, it is therefore enough to argue that the vector of martingales $\left(  {M}_j^{X},  {M}_j^{Y},\; j \in [m] \right)$ has a stochastic integral representation as the local martingale component in \eqref{eq:limitsde}. Once this is achieved, using the uniqueness in law of a process satisfying SDE \eqref{eq:limitsde} proved in Theorem \ref{lem:sdesoln} below, every weak limit must be the same and given by the solution of \eqref{eq:limitsde}. 
	
	To carry this out carefully, start be expressing the processes ${M}_j^{X,n}, {M}_j^{Y,n}$, $j\in [m]$, as martingales with respect to natural filtrations. Fix $n \in \NN$. Recall the PPP $\chi$ on $(0,\infty) \times \oD$ from the beginning of the proof of Lemma \ref{lem:1pointlimit}. Extend the PPP by decorating each atom of $\chi$ by an independent vector of length $(m+1)$, $\left( \theta, H_1, \ldots, H_m \right)$ where we sample $\theta \sim \nu_0$, and $\left(  H_j,\; j \in [m] \right)$ are i.i.d. Bernoulli$(p)$ picks, independent of $\theta$. This produces a PPP $\overline{\chi}$ on $[0, \infty)\times \oD \times [0,2\pi] \times \{0,1\}^m$. Choose a suitable probability space $\left(\Omega, \fil_{\infty}, \mathcal{P} \right)$ that supports $\overline{\chi}$. Let $\left(\fil_t,\; t\ge 0 \right)$ be the natural right continuous filtration generated by the process $\left( \overline{\chi}_t,\; t\ge 0   \right)$ where $\overline{\chi}_t$ is the restriction of $\overline{\chi}$ to $[0,t]\times \oD \times [0,2\pi] \times \{0,1\}^m$. Note that, as opposed to Definition \ref{defn:smooshingcards}, in this proof we attach the random angle and Bernoulli variables whether or not there are cards ``under the palm''. They simply do not influence the motion of the cards unless the cards are under the palm. Enumerate the countably many atoms of $\overline{\chi}$ by $\left(  (t_i, w_i, \theta_i, H_1(i), \ldots, H_m(i)),\; i \in \NN  \right)$ where $t_1 < t_2 < \ldots$. Then, on our sample space above we have the following expressions:
	\[
	\begin{split}
		{M}_j^{X,n}(t) &= \frac{1}{\sqrt{n}} \sum_{i: t_i \le t} H_j(i) \cos(\theta_i) 1\left\{  \normU{Z_j^{(n)}(t_i-) - w_i} \le \delta \right\},\\
		{M}_j^{Y,n}(t) &= \frac{1}{\sqrt{n}} \sum_{i: t_i \le t} H_j(i)\sin(\theta_i) 1\left\{  \normU{Z_j^{(n)}(t_i-) - w_i} \le \delta \right\}.
	\end{split}
	\]
	Here, as before, $Z_j^{(n)}=\left( X_j^{(n)}, Y_j^{(n)}  \right)$. 
	
	Let $Z^{(n)}$ denote the $2m$ dimensional vector of $\left( Z_j^{(n)},\; j \in [m]   \right)$. Recall the $2m\times 2m$ dimensional matrix $B(z_1, \ldots, z_m)$ from \eqref{eq:limitsde}. For $z=(z_1=(x_1, y_1), \ldots, z_m=(x_m,y_m))\in \rr^{2m}$, label the elements of $B$ by $B_{x_j,x_k}(z)$, $B_{x_j,y_k}(z)$, or $B_{y_j,y_k}(z)$, for $j,k \in [m]$, by a natural correspondence. 
	
	\begin{lemma}\label{lem:compensators} In the filtered probability space described above each ${M}_j^{X,n}$ and ${M}_j^{Y,n}$ is an $\left( \fil_t \right)$ martingale. Moreover, for all $(j,k)\in [m]^2$, the following processes are also $\left( \fil_t \right)$  martingales:
		\eq\label{eq:compcov}
		\begin{split}
			\xi_{j,k}^{X,X,n}(t)&:= {M}_j^{X,n}(t) {M}_k^{X,n}(t) - \int_0^t B_{x_j,x_k}\left(Z^{(n)}(s) \right)ds.\\
			\xi_{j,k}^{X,Y,n}(t)&:= {M}_j^{X,n}(t) {M}_k^{Y,n}(t) - \int_0^t B_{x_j,y_k}\left(Z^{(n)}(s) \right)ds.\\
			\xi_{j,k}^{Y,Y,n}(t)&:= {M}_j^{Y,n}(t) {M}_k^{Y,n}(t) - \int_0^t B_{y_j,y_k}\left(Z^{(n)}(s) \right)ds.
		\end{split}
		\en
	\end{lemma}
	
	\begin{proof}[Proof of Lemma \ref{lem:compensators}] We start by arguing that ${M}_j^{X,n}$ and ${M}_j^{Y,n}$ are martingales. For every $\omega\in \Omega$, the processes ${M}_j^{X,n}$ and ${M}_j^{Y,n}$ are stochastic integrals of predictable integrands: for $\theta \in [0, 2\pi]$ and $h_j\in \{0,1\}$, $j\in [m]$, 
		\[
		\begin{split}
			g^{X,n}_j(t,\omega)&=\frac{1}{\sqrt{n}} h_j\cos(\theta)1\left\{ \normU{Z_j^{(n)}(t-)- w} \le 1  \right\}\quad \text{and}\\
			g^{Y,n}_j(t,\omega)&=\frac{1}{\sqrt{n}} h_j\sin(\theta)1\left\{ \normU{Z_j^{(n)}(t-)- w} \le 1  \right\}, \quad \omega \in \Omega,
		\end{split}
		\]
		with respect to the Poisson random measure $\overline{\chi}$ (see \cite[Chapter II, Section 1]{JS}). Then, the claim follows from \cite[Chapter II, Lemma 1.21]{JS}, since, the predictable compensator of the processes are given by (respectively)
		\[
		\begin{split}
			\frac{1}{\sqrt{n}}\int_0^t n p \nu_0\left( \cos(\theta) \right) \Ar{Z_j^{(n)}(s-) + U}ds \quad \text{and}\\
			\frac{1}{\sqrt{n}}\int_0^t n p \nu_0\left( \sin(\theta) \right) \Ar{Z_j^{(n)}(s-) + U}ds.
		\end{split}
		\]
		Both expressions above are zero since $\nu_0(\cos\theta)=0=\nu_0(\sin\theta)$. 
		
		For the reader who might be uncomfortable with the stochastic calculus for Poisson processes, simply replace the Poisson process by a discrete time process with independent increments to derive the above conclusion ``by hand''.  This is true for the argument below as well.

		For the processes listed in \eqref{eq:compcov}, let us argue the martingale property of the first process in the display and leave the rest of the similar arguments for the reader. Consider the process $\xi_{j,k}^{X,X,n}$. Since ${M}_j^{X,n}(t)$ and ${M}_k^{X,n}(t)$ are both martingales, we simply need to argue  that the predictable compensator for the product of the two process at time $t$ is exactly $\int_0^t B_{x_j,x_k}\left(Z^{(n)}(s) \right)ds$. 
		However, since ${M}_j^{X,n}(t)$ and ${M}_k^{X,n}(t)$ are both stochastic integrals of predictable integrands with respect to a Poisson random measure, the predictable compensator up to time $t$ is given by the integral of the product of the integrands with respect to the intensity measure:
		\[
		\int_0^t \left( p 1\{j=k\} + p^2 1\{j\neq k \}  \right) \sigma^2 \Ar{\left\{ Z_j^{(n)}(s-) + U   \right\}\cap \left\{ Z_k^{(n)}(s-) + U \right\} }ds.
		\]
		The above is, of course, exactly equal to $\int_0^t B_{x_j,x_k}\left(Z^{(n)}(s) \right)ds$. 
	\end{proof}

	Returning to the proof of Theorem \ref{thm:jpoint}, recall that $\mathcal{C}^{(4m)}[0, \infty)$, the space of continuous functions from $[0,\infty)$ to $\rr^{4m}$. Endow the space with a right-continuous natural filtration. We will use this as our sample space. Consider this sample space along with a probability measure that is any weak limit obtained from the joint weak convergence of the vector of processes in \eqref{eq:limittilde} to the processes in \eqref{eq:limitproc}. 
	
	It follows by localization that, under any weak limit in $\mathcal{C}^{(4m)}[0,\infty)$, each ${M}^X_j, {M}^Y_j$ is a continuous local martingale such that each of the following is also a continuous local martingale:
	\[ 
	\begin{split}
		\xi_{j,k}^{X,X}(t)&:= {M}_j^{X}(t) {M}_k^{X}(t) - \int_0^t B_{x_j,x_k}\left(Z(s) \right)ds.\\
		\xi_{j,k}^{X,Y}(t)&:= {M}_j^{X}(t) {M}_k^{Y}(t) - \int_0^t B_{x_j,y_k}\left(Z(s) \right)ds.\\
		\xi_{j,k}^{Y,Y}(t)&:= {M}_j^{Y}(t) {M}_k^{Y}(t) - \int_0^t B_{y_j,y_k}\left(Z(s) \right)ds.
	\end{split}
	\]
	Here $Z_j(\cdot)=\left( X_j(\cdot), Y_j(\cdot) \right)$ and $Z(\cdot)=\left( Z_1(\cdot), \ldots, Z_m(\cdot) \right)$.
	
	We now use \cite[Chapter 3, Theorem 4.2]{KS91} on the representation of continuous local martingales as stochastic integrals. According to this result, on a possibly extended probability space, one can find a $2m$ dimensional Brownian motion $\left( \beta_1, \ldots, \beta_{2m}   \right)$ such that for each $j\in [m]$ we have
	\[
	\begin{split}
		{M}_j^{X}(t) &= \int_0^t A_{X_j}\left( Z_1(s), \ldots, Z_m(s)  \right)\cdot d\beta(s) \quad \text{and}\\
		{M}_j^{Y}(t) &= \int_0^t A_{Y_j}\left( Z_1(s), \ldots, Z_m(s)  \right) \cdot d\beta(s).
	\end{split}
	\]
	This settles the local martingale component in the SDE representation \eqref{eq:limitsde}. That the finite variation components are given by local times follow from Lemma \ref{lem:1pointlimit}. Finally, uniqueness in law from Theorem \ref{lem:sdesoln} below completes the proof. 
\end{proof}

\begin{thm}\label{lem:sdesoln} Fix arbitrary initial points $z_1, \ldots, z_m$ in $[0,1]^2$. Under Assumption  \ref{asmp:unbiased} and when $p \in (0,1)$, for any $m \ge 1$, there is a pathwise unique strong solution to the stochastic differential equation \eqref{eq:limitsde}, starting at $\left(z_1, \ldots, z_m\right)$, under which the process is strong Markov. In particular, the law of such a solution is unique. 
\end{thm}

The proof requires the following lemma.

\begin{lemma}\label{lem:unifellip} The diffusion matrix  $B(z_1, \ldots, z_m)$ is uniformly elliptic over $[0,1]^{2m}$.
\end{lemma}

\begin{proof} Since $B=F\otimes \Sigma$ is the Kronecker product of $F$ and $\Sigma$, the $2m$ eigenvalues of $B$ are the pairwise product of the $m$ eigenvalues of $F$ and the two eigenvalues of $\Sigma$. Since the eigenvalues of $\Sigma$ are both $\sigma^2$, they are both positive by Assumption \ref{asmp:unbiased}. Therefore, to prove the lemma, it suffices to show uniform ellipticity for the matrix $F(z_1, \ldots, z_m)$.
	
	Consider any $\xi=\left( \xi_1, \ldots, \xi_m   \right) \in \rr^{m}$ and note that 
	\[
	\begin{split}
		\xi' &F(z_1, \ldots, z_m) \xi \\
		&= \sum_{i\in [m]} \sum_{j \in [m]} p^2  \xi_i \xi_j\Ar{\{ z_i + U\} \cap \{ z_j+U \}} + (p-p^2) \pi \delta^2 \sum_{i\in [m]} \xi_i^2\\
		&= p^2 \int_{\rr^2} \left[  \sum_{i=1}^m \sum_{j=1}^m \xi_i \xi_j 1\{ v\in  z_i + U\} 1\{ v\in  z_j+U \}    \right] dv + p(1-p) \pi \delta^2\norm{\xi}^2\\
		&=p^2\int_{\rr^2} \left(  \sum_{j=1}^m \xi_j 1\{ v\in z_j + U  \}   \right)^2 dv +  p(1-p) \pi \delta^2 \norm{\xi}^2 \ge p(1-p) \pi \delta^2 \norm{\xi}^2.
	\end{split}
	\]
	Since $p(1-p) > 0$ this proves uniform ellipticity. 
\end{proof}

\begin{proof}[Proof of Theorem \ref{lem:sdesoln}] We verify the assumptions of \cite[Theorem 4.3]{Ramanan06} which has been proved for the so-called Extended Skorokhod Problem (ESP). In particular, it holds for the case of Skorokhod problems. 
	
	Our Skorokhod map is coordinatewise given by \eqref{eq:explicitskorokhod}. Therefore, it is well-defined and Lipschitz. Therefore, it suffices to check Assumption 4.1 (1) in \cite{Ramanan06}. Since the drift is zero, we need to only check that the map $(z_1, \ldots, z_m) \mapsto A\left(z_1, \ldots, z_m\right)$, as a function on $[0,1]^{2m}$, is Lipschitz. 
	By \cite[Lemma 21.10]{schilling14} and the uniform ellipticity condition from Lemma \ref{lem:unifellip} it suffices to check that the map $(z_1, \ldots, z_m) \mapsto B(z_1, \ldots, z_m)$ is Lipschitz. This, in turn, follows from checking via Lemma \ref{lem:areaint} that the map $(z_1, \ldots, z_m) \mapsto F_{ij}(z_1, \ldots, z_m)$ for each $(i,j)$ pair is Lipschitz which follows from the convexity of the function $\varphi$.  
\end{proof}

\section{Estimates on mixing time of shuffling for the jump diffusion}\label{sec:coupling}

We now define the limiting lazy {\gands} model. Let us recall the diffusion model from Section \ref{sec:diffusionlimit}. Consider a suitable probability space $\left( \Omega, (\mathcal{F}_t)_{t\ge 0}, P \right)$ with the usual filtration that supports $2m$ many standard linear Brownian motions $\left( \overline{\beta}_1, \ldots, \overline{\beta}_{2m}  \right)$ and an independent PPP on $(0, \infty)\times \oD$ with rate given by the product Lebesgue measure on $(0, \infty)$ and the uniform probability distribution on $\oD$. That is the atoms of the PPP can be arranged as $(t_i, w_i)$, $i\in \NN$, where $0 < t_1 < t_2 < \ldots$ are the jumps of a Poisson process of rate one and the sequence $(w_i,\; i \in \NN)$ is {i.i.d.}, sampled uniformly from $\oD$, independently of $(t_i,\; i \in \NN)$.

Suppose the initial values $z_1=(x_1, y_1), \ldots, z_m=(x_m, y_m)$.
Let $t_0=0$ and define $Z(t)=\left(  Z_1(t), \ldots, Z_m(t)   \right)$, $t \in [0, t_1)$, as the solution of SDE \eqref{eq:limitsde} with initial conditions $X_j(0)=x_j$ and $Y_j(0)=y_j$, $j \in [m]$, and the Brownian motions given by $\overline{\beta}_j$, $j\in [2m]$. Then, inductively, for $i=1,2,\ldots$, on $[t_i, t_{i+1})$, condition on $\mathcal{F}_{t_i}$, define initial conditions
\[
z_j(t_i):=G^{w_i}\left( Z_j(t_i-)   \right),\quad j\in [m],
\]
and let $Z(t+ t_i)$, $t\in [0, t_{i+1}-t_i)$ be the solution of SDE \eqref{eq:limitsde} with initial points $z_j(t_i)$, $j \in [m]$, and  the Brownian motions given by $\beta^{(i)}_j(t)=\overline{\beta}_j(t_i+t)-\overline{\beta}_j(t_i)$. This gives us a jump-diffusion process $Z(t)=\left(  Z_1(t), \ldots, Z_m(t)   \right)$, $t \in [0, \infty)$ with RCLL paths adapted to $\left( \Omega, (\mathcal{F}_t)_{t\ge 0}, P \right)$. The process is clearly strong Markov. Let $\jdiffdist_m(z_1, \ldots, z_m)$ denote the law of the jump-diffusion process described above starting from the initial points $(z_1, \ldots, z_m)\in D^m$.

\begin{lemma}\label{lem:exchangeable3}
	Assumption \ref{asmp:exchangeable2} in Section \ref{sec:discrete} holds for the jump-diffusion process $Z$. 
\end{lemma}

\begin{proof} Assumption \ref{asmp:exchangeable2} is clearly true for the discrete {\gands} model without gathering as in the beginning of Section \ref{sec:diffusionlimit}. Hence, by taking weak limit, it is true for the diffusion satisfying \eqref{eq:limitsde}. Stopping the diffusion at an independent exponential time and gathering at an independently chosen point in $D$ preserves the property. 
	Now, by iterating the argument, the statement is true for all intervals $[t_i, t_{i+1})$, $i \in \NN$, and, therefore, over the entire $[0, \infty)$.
\end{proof}



\begin{proof}[Proof of Theorem \ref{thm:mixingtime}] 
	
	In order to employ Theorem \ref{thm:mixdiscrete}, it suffices to find a $\zeta^*$ that stochastically dominates $\overline{\tau}^{}_{xy}=\overline{\tau}^{m}_{xy}$ from \eqref{eq:whatistaubar} for any $x\neq y$. Without loss of generality, take $m=2$.

	The strategy is the following. Consider $\norm{Z_1(t)-Z_2(t)}$ during $[0, t_1)$. Ignore the possibility that this norm hits zero during this interval. Consider 
	\[
	\left\{ Z_1(t_1-) + U  \right\} \cap \left\{ Z_2(t_1-) + U  \right\}.
	\]
	The area of this set is given by Lemma \ref{lem:areaint}. $Z_1(t_1)=Z_2(t_1)$ if the random point $w_1$ lies in this set (since the gathering will place both cards at $w_1$). In this event we get $\overline{\tau}_{12} \le t_1$, otherwise we restart. Hence, at each $t_i$ we toss a coin that indicates if $w_i \in \left\{ Z_1(t_i-) + U  \right\} \cap \left\{ Z_2(t_i-) + U  \right\}$. Suppose we bound this probability from below by $\mathfrak{p}>0$ (say), irrespective of $z_1, z_2$. Then, by the strong Markov property, each such coin toss is independent, and $\overline{\tau}_{12}$ is stochastically dominated by $t_J$, where $J$ is a geometric random variable with rate $\mathfrak{p}$. $t_J$ is the sum of a random number of exponential one random variables. Such a random variable has a density everywhere on $[0,\infty)$, in particular, condition \eqref{eq:exptail} is satisfied where $\alpha$ is simply the density at zero. Now, given $J=j$, $t_J$ is a gamma random variable with mean $j$ and scale $1$. For all $j \ge 2$, this random variable has density zero at the origin. Given $J=1$, $t_J$ is exponential with rate one, which has a density one at the origin. Thus, the density at the origin of $t_J$ is $1\cdot P(J=1)=\mathfrak{p}$. Thus, we can take $\alpha=\mathfrak{p}$ in \eqref{eq:exptail}. 
	Hence, by Theorem \ref{thm:mixdiscrete}, $\tmix(\epsilon)=\frac{1}{\mathfrak{p}} \log(m/\epsilon)$. 
	
	What remains is to find such a $\mathfrak{p}$. This is done in the rest of this article.
\end{proof}

Consider any of the intervals $[t_i, t_{i+1})$, condition on $\mathcal{F}_{t_i}$ and shift time $t \mapsto t - t_i$, for $t\in [t_i, t_{i+1})$. During this interval the jump-diffusion $Z$ is simply a diffusion stopped at an independent rate one exponential time $\varrho:=t_{i+1}-t_i$. Thus, by the strong Markov property, we can assume that $i=0=t_i$ and $\varrho:=t_1$ is an exponential one random variable, independent of the diffusion $Z(t)$, $t\in [0, \varrho)$.

Express this stopped diffusion $Z(\cdot)=\left(Z_1(\cdot), Z_2(\cdot) \right)$, where $Z_1=(X_1, Y_1)$, $Z_2=(X_2, Y_2)$, which is a solution of \eqref{eq:limitsde}, in the following way: 
\begin{eqnarray*}
	X_1(t) &=& W_1(t) + L_1^{X,0}(t) - L_1^{X,1}(t),\\
	Y_1(t) &=&  B_1(t) + L_1^{Y,0}(t) - L_1^{Y,1}(t),\\
	X_2(t) &=&  W_2(t) + L_2^{X,0}(t) - L_2^{X,1}(t),\\
	Y_2(t) &=& B_2(t) + L_2^{Y,0}(t) - L_2^{Y,1}(t).
\end{eqnarray*}
Here $\left(W_1, B_1, W_2, B_2\right)$ is a four-dimensional continuous semimartingale process such that each coordinate process is marginally distributed as a Brownian motion with constant diffusion coefficients $p\pi\delta^2\sigma^2$ and initial values $x_1, y_1, x_2, y_2$, respectively. But they are not all independent. This is because, in the discrete model, the cards that are under the palm tend to move together during the spread moves leading to a positive correlation between their increments, as compared to zero correlation when not under the palm. 


Consider the two processes 
\[
\beta_1^0(t) := W_1(t) - W_2(t), \quad \beta_2^0(t):= B_1(t) - B_2(t).
\]
Let $\left[X,Y\right](t)$ denote the mutual covariation between two continuous semimartingales $X$ and $Y$ over the time interval $[0,t]$. By \eqref{eq:twobytwo}, 
\[
\begin{split}
	[\beta_1^0, \beta_2^0](t)=0,\quad [\beta_1^0, \beta_1^0](t)=[\beta_2^0, \beta_2^0](t)=2\sigma^2\int_0^t\left(p\pi \delta^2 - F_{12}\left( Z_1(s), Z_2(s) \right)   \right)ds.
\end{split}
\]
Let 
\[
\Gamma(t)= 2\sigma^2\int_0^t\left(p\pi \delta^2 - F_{12}\left( Z_1(s), Z_2(s) \right)   \right)ds.
\]
Since $0\le F_{12}\left( z_1, z_2 \right)\le p^2 \Ar{U}=p^2\pi \delta^2$ for all $(z_1, z_2)\in D^2$, then 
\eq\label{eq:bndGamma}
2\sigma^2p\pi \delta^2 \ge \Gamma'(t) \ge 2\sigma^2 p(1-p)\pi \delta^2 > 0, \quad \text{and}\quad \Gamma(t) \ge 2\sigma^2 p(1-p)\pi \delta^2 t.
\en
Hence $\Gamma$ is strictly increasing and $\lim_{t\rightarrow \infty} \Gamma(t)=\infty$ almost surely. By Knight's Theorem (see \cite[page 179]{KS91}), on the same probability space, there exists a pair of independent standard Brownian motions $\beta_1,\beta_2$ such that 
\[
\beta^0_1(t) = \beta_1\left( \Gamma(t) \right), \quad \beta^0_2(t) = \beta_2\left( \Gamma(t)  \right).
\]

Let $R(t)= \sqrt{\beta^2_1(t) + \beta^2_2(t)}$. Then $R$ is a $2$-dimensional Bessel process starting from $\norm{z_1-z_2}$. Now, $X_1$ and $X_2$ are doubly reflected Brownian motions with continuous noises $W_1$ and $W_2$, respectively (as in Definition \ref{defn:skorokhodmap}). The maps $W_1\mapsto X_1$ and $W_2\mapsto X_2$ are Lipschitz in the locally uniform metric on $\mathcal{C}[0, \infty)$. See \cite[Remark 4.2 (ii)]{Saisho}. Therefore $\left( X_1(t)- X_2(t)  \right)^2 \le \left( W_1(t) - W_2(t)  \right)^2$. Similarly $\left( Y_1(t)- Y_2(t)  \right)^2 \le \left( B_1(t) - B_2(t)  \right)^2$. Hence 
\[
\norm{Z_1(t) - Z_2(t)} \le \sqrt{\left(\beta^0_1(t)\right)^2 + \left(\beta^0_2(t)\right)^2}= R\left(  \Gamma(t) \right), \quad 0\le t < \varrho,
\]
where $\varrho$ is the independent exponential one random variable.

At $\varrho$ pick $w$ uniformly from $\oD$ independently of the process $Z(t)$, $t\in [0, \varrho)$, and $\varrho$ itself. 
Recall that $U$ is the closed disc of radius $\delta$ around the origin. Consider the Bernoulli random variable
\eq\label{eq:bernoullis}
\begin{split}
	\chi_1&:=1\left\{  w \in \{ Z_1(\varrho-) + U \} \cap \{ Z_2(\varrho-) + U  \} \right\}.
\end{split}
\en
Given $Z(t)$, $t\in [0, \varrho)$, the probability that $\{\chi_1=1\}$ is given by 
\[
\frac{1}{\Ar{\oD}} \Ar{ \{ Z_1(\varrho-) + U \} \cap \{ Z_2(\varrho-) + U  \}}= \frac{1}{\Ar{\oD}} \varphi\left( \norm{Z_1(\varrho-) - Z_2(\varrho-)} \right),
\]
by Lemma \ref{lem:areaint}. Since $\varphi$, defined in Lemma \ref{lem:areaint}, is decreasing the above expressions are bounded below by 
\eq\label{eq:plbnd}
\frac{1}{\Ar{\oD}} \varphi\left( R(\Gamma(\varrho)\right).
\en


Hence, integrating with respect to the law of $Z(t)$, $t\in [0, \varrho)$, and $\varrho$, we get
\[
\begin{split}
	\Ar{\oD}\E\left( \chi_1 \right) &\ge \E\left[ \int_0^\infty e^{-t} \varphi\left( R(\Gamma(t))\right)dt \right]\\
	&=  \E\left[ \int_0^\infty e^{-\Gamma^{-1}(u)} \varphi\left( R(u)\right)\left( \Gamma^{-1}(u) \right)' du \right], \quad u=\Gamma(t).
\end{split}
\]
From the bounds in \eqref{eq:bndGamma} we get
\[
\frac{1}{2\sigma^2p(1-p)\pi \delta^2}\ge \left(\Gamma^{-1}(u)\right)' \ge \frac{1}{2\sigma^2 p\pi \delta^2}, \quad \text{and}\quad \Gamma^{-1}(u) \le \frac{u}{2\sigma^2 p(1-p)\pi \delta^2}. 
\]
Hence
\[
\begin{split}
	\Ar{\oD}\E\left( \chi_1 \right) &\ge \frac{1}{2\sigma^2p\pi \delta^2} \E\left[ \int_0^\infty e^{-c u}\varphi(R(u)) du \right],\quad c=\frac{1}{2\sigma^2 p(1-p)\pi \delta^2},\\
	&= \frac{1}{2\sigma^2p\pi \delta^2}  \int_0^\infty e^{-c u} \E\left[\varphi(R(u)) \right] du\\
	&= \frac{1}{2\sigma^2p\pi \delta^2}  \int_0^\infty e^{-c u} Q^{2}_{\norm{z_1-z_2}}\left[\varphi(R(u)) \right] du,
\end{split}
\]
where $Q^2_x$ is the law of a two dimensional Bessel process starting from $x$.

Now $\norm{z_1-z_2}\le \sqrt{2}$, the diameter of $D$. It follows from additivity of squared Bessel processes (see \cite{ShigaWata73}) that the law of $R(u)$, under $Q_{\norm{z_1-z_2}}^2$, is stochastically dominated by the law of $R(u)$, under $Q_{\sqrt{2}}^2$. Using the fact that $\varphi$ is decreasing we get 
\[
Q_{\norm{z_1-z_2}}^2\left[\varphi\left( R(u)  \right)\right] \ge Q_{\sqrt{2}}^2\left[\varphi\left( R(u)  \right)\right].
\]
Combining all the bounds and noting that $\Ar{\oD}\le (1+2\delta)^2$, we get
\eq\label{eq:estimatebern}
\begin{split}
	\E\left( \chi_1\right) \ge \frac{1}{2\sigma^2 p \pi \delta^2 (1+2\delta)^2} \int_0^\infty \exp\left( -\frac{u}{2\sigma^2 p(1-p)\pi \delta^2}\right)Q_{\sqrt{2}}^2\left[\varphi\left( R(u)  \right)\right]du.
\end{split}
\en
Notice that this is a lower bound that is independent of the starting position of the diffusion.

To estimate the last expression we express it back in terms of planar Brownian motion. Let $v=(v_1, v_2)\in \rr^2$ be arbitrary and, as before, let $c=\left(2\sigma^2 p(1-p)\pi \delta^2\right)^{-1}$. Let $V=(V_1, V_2)$ be a planar Brownian motion, starting from $v=(v_1, v_2)$. Then 
\[
\begin{split}
	\int_0^\infty &\exp\left( -\frac{u}{2\sigma^2 p(1-p)\pi \delta^2}\right)Q_{\norm{v}}^2\left[\varphi\left( R(u)  \right)\right]du= \int_0^\infty e^{-c u} \E_v\left[  \varphi\left( \norm{V}(u) \right) \right]du.
\end{split}
\] 
The last expression is the $c$ resolvent (sometimes called the $c$ potential) operator for the generator of planar Brownian motion and is known explicitly. See \cite[page 93]{schilling14}:
\[
\begin{split}
	\int_0^\infty e^{-c u} \E_v\left[  \varphi\left( \norm{V}(u) \right) \right]du.&=\frac{1}{\pi} \int_{\rr^2} K_0\left( \sqrt{2c} \norm{y} \right) \varphi(\norm{v-y})dy\\
	&=\frac{1}{\pi} \int_{z:\norm{z}<2\delta} K_0\left( \sqrt{2c} \norm{v-z} \right) \varphi(\norm{z})dz,
\end{split}
\]
where $K_0$ is the modified Bessel function of the second kind. The second equality above is due to the fact that $\varphi(r)=0$ for $r\ge 2\delta$.

Now $K_0$ is a decreasing function of its argument (see \cite[page 374]{AbraSteg}). Thus
\[
\begin{split}
	\frac{1}{\pi} &\int_{z:\norm{z}<2\delta} K_0\left( \sqrt{2c} \norm{v-z} \right) \varphi(\norm{z})dz \ge K_0\left( \sqrt{2c} (\norm{v} + 2\delta) \right)\frac{1}{\pi} \int_{z:\norm{z}<2\delta}  \varphi(\norm{z})dz.
\end{split}
\]
We now put $\norm{v}=\sqrt{2}$ and evaluate the following integral:
\[
\begin{split}
	\frac{1}{\pi}&\int_{z:\norm{z}<2\delta}  \varphi(\norm{z})dz =2 \int_0^{2\delta} r\varphi(r) dr\\
	&=4\delta^2\int_0^{2\delta} r \arccos\left(  \frac{r}{2\delta} \right)dr - \int_0^{2\delta} r^2\sqrt{4\delta^2-r^2}dr\\
	&=16\delta^4\left[ \int_0^{1} s \arccos\left( s \right) ds - \int_0^{1} s^2 \sqrt{1-s^2} ds\right],\; s=\frac{r}{2\delta},\\
	&= 16\delta^4 \left[ \frac{\pi}{8} - \frac{\pi}{16}  \right]= \pi \delta^4.
\end{split}
\]

Combining all the bounds, 
\eq\label{eq:mathfrakp}
\E\left( \chi_1\right) \ge \mathfrak{p} := \frac{\delta^2}{2p\pi\sigma^2 (1+2\delta)^2} K_0\left( \frac{\sqrt{2}+2\delta}{\sigma\delta \sqrt{\pi p (1-p)}}   \right).
\en
Therefore, at $\varrho=t_1$, the probability that we gather cards $1$ and $2$ is at least $\mathfrak{p}$, irrespective of the starting positions of the two cards. This gives an upper bound on $\mathfrak{p}$ and completes the proof of Theorem \ref{thm:mixingtime}.

\section*{Acknowledgements}
It is our pleasure to thank Lauren Bandklayder, Marc Coram, Emanuelle Gouillart, Kimberly Kinateder, Mark Perlman, and Graham White for their help with this project over the years. Many thanks to Keith Fife, Yuqi Huang, and Max Goering at the University of Washington for their help with the simulations. 
The first author was supported in part by NSF grant DMS-1208775.
The second author was supported in part by NSF grants DMS-1308340 and DMS-1612483.

\bibliographystyle{plain} 
\bibliography{smooshing.bib}

\end{document}